\documentclass[11pt]{article}
\usepackage{amsfonts}
\usepackage{amsmath}
\usepackage{amssymb}
\usepackage{epsfig}
\usepackage{enumerate}
\usepackage{citesort}
\usepackage{url}
\title{Integral equations requiring small numbers of Krylov-subspace
iterations for two-dimensional penetrable scattering problems}
\author{Yassine Boubendir, Oscar Bruno, David Levadoux, 
Catalin Turc\\ \small 
New Jersey Institute of Technology, Caltech, ONERA France, New Jersey Institute of Technology\\
\small boubendi@njit.edu,\ obruno@caltech.edu,\ david.levadoux@onera.fr,\ catalin.c.turc@njit.edu}

\newtheorem{theorem}{Theorem}[section]
\newtheorem{lemma}[theorem]{Lemma}

\newtheorem{remark}[theorem]{Remark}
\newenvironment{proof}{\hspace{0.5cm} {\bf Proof.}}
{$\quad {}_\blacksquare$\vspace{0.3cm}}

\date{}

\begin{document}
\maketitle
\begin{abstract}
This paper presents a class of boundary integral equations for the
  solution of problems of electromagnetic and acoustic scattering by
  two dimensional homogeneous penetrable scatterers with smooth
  boundaries.
  The new integral equations, which, as is established in this paper,
  are uniquely solvable Fredholm equations of the second kind, result
  from representations of fields as combinations of single and double
  layer potentials acting on appropriately chosen regularizing
  operators.  As demonstrated in this text by means of a variety of
  numerical examples (that resulted from a high-order Nystr\"om
  computational implementation of the new equations), these
  ``regularized combined equations'' can give rise to important
  reductions in computational costs, for a given accuracy, over those
  resulting from previous boundary integral formulations for
  transmission problems.  \newline \indent $\mathbf{Keywords}$:
  electromagnetic scattering, transmission problems, Combined Field
  Integral Equations, pseudo-differential operators, regularizing
  operator.
\end{abstract}

\section{Introduction\label{intro}}

Owing mainly to the dimensional reduction and the absence of
dispersion errors inherent in boundary integral formulations,
scattering solvers based on boundary integral equations can be
significantly more efficient than solvers resulting from direct
discretization of the corresponding Partial Differential Equations
(PDE). In view of the large dense systems of linear equations that
result from integral equation formulations, however, integral-equation
solvers often rely on use of Krylov subspace iterative solvers in
conjunction with fast algorithms~\cite{br-k1,bles,r3} for evaluation
of matrix vector products.  Since the numbers of iterations required
by Krylov subspace solvers for a given system of equations depends
greatly on the spectral character of the system matrix, the efficiency
of boundary iterative integral solvers hinges on the spectral
properties of the integral formulations used.

In this paper we introduce a new class of boundary integral equations
for the solution of electromagnetic and acoustic transmission problems
(that is, problems of diffraction and scattering by penetrable
scatterers) for two dimensional homogeneous penetrable scatterers with
smooth boundaries. The new formulations, which, as established in this
paper, are expressed in terms of uniquely solvable Fredholm equations
of the second kind, result from representations of fields as
combinations of single and double layer potentials acting on
regularizing operators analogous to those used
in~\cite{bept,br-turc,br-lintner-3d,br-lintner,AlougesLevadoux,BorelLevadouxAlouges,Antoine,AntoineX,AntoineX,Levadoux,Levadoux1,Levadoux2} for non-penetrable
scatterers.  As demonstrated in this text by means of a variety of
numerical examples, for a given accuracy, the new ``regularized
combined equations'' can give rise to significant reductions in
iteration numbers, and therefore, in overall computational costs, over
those resulting from previous boundary integral formulations.

Uniquely solvable formulations for transmission scattering problems
have been available for quite some time~\cite{muller}.  For example, a
class of such integral equations results from representation of the
fields inside and outside the dielectric scatterer by means of Green's
formula: linear combinations the interior and exterior boundary values
of the fields and their normal derivatives can be used to produce
various types of pairs of boundary integral equations with two unknown
functions, including: (1)~A family of integral equations {\em of the
  second kind} (which feature multiples of the identity operator and
linear combinations of compact operators, including the single and
double layer operators as well as suitable {\em compact} differences
of hypersingular integral
operators~\cite{muller,KressRoach,KittapaKleinman,rokhlin-dielectric});
and (2)~Integral equations {\em of the first kind with positive
  principal part} which include non-compact hypersingular
operators. Equations of the type~(1) and~(2) are used most frequently
as part of integral solvers for wave transmission problems.

(It is also possible to express the transmission problem in terms of a
{\em single} uniquely solvable integral equation; the
paper~\cite{KleinmanMartin} presents a comprehensive discussion in
these regards. It is useful to mention here that, while single
transmission integral equations contain half as many unknowns as the
corresponding systems, they require operator compositions which, in
the context of iterative solvers considered in this paper, may lead to
equal or even higher computing costs than the system formulations such
as those mentioned above.)

This paper presents novel {\em regularized} integral equation
formulations for the transmission problem. Unlike the second-kind
formulations introduced earlier, which rely on cancellation of
hypersingular terms in the integral kernels, the present approach
produces second-kind formulations via {\em regularization of
  hypersingular operators by means of suitable regularizing
  operators}~\cite{bept,br-turc,br-lintner-3d,br-lintner,Levadoux1,Levadoux2}.  In
particular, use of regularizing operators given by layer-potentials
with {\em complex wavenumbers} results in second-kind Fredholm
equations with improved spectral properties.  Similar regularized
equations result from use of regularizing operators expressed in terms
of {\em Fourier operators} whose symbols have the same high-frequency
asymptotics as the corresponding layer-potential regularizers
mentioned above.

In order to demonstrate the properties of the new equations we
introduce Nystr\"om implementations of the previous and new
formulations considered. These implementations, which are based on the
methods introduced in~\cite{kusmaul,martensen,KressH,turc1}, include
elements such as approximation via global trigonometric polynomials,
splitting of integral kernels into singular and smooth components, and
explicit quadrature of products of logarithmically singular terms and
trigonometric polynomials.

This paper is organized as follows: Section~\ref{cfie} reviews the
classical systems of integral equations for the solution of
transmission scattering problems, it introduces the new regularized
formulations, and it establishes their unique solvability and
second-kind Fredholm character in the appropriate functional
spaces. The reductions in iteration numbers and computing times that
result from the new formulations are demonstrated in
Section~\ref{prec}.

\section{Acoustic transmission problems\label{cfie}}

We consider the problem of evaluation of the time-harmonic fields that
arise as an incident field $u^{inc}$ impinges upon a homogeneous
dielectric scatterer which occupies a bounded region
$\Omega_2\subset\mathbb{R}^2$. Calling $\Omega_1 =
\mathbb{R}^2\setminus \Omega_2$ the region exterior to the obstacle,
in what follows $\epsilon_1$ and $\epsilon_2$ denote the electric
permeabilities of the materials in regions $\Omega_1$ and $\Omega_2$,
respectively; the permeabilities of both dielectrics is assumed to
equal $\mu_0$, the permeability of vacuum. Letting $u^1$ and $u^2$
denote the scattered field in $\Omega_1$ and the total field in
$\Omega_2$, respectively, calling $\Gamma$ the boundary between the
domains $\Omega_1$ and $\Omega_2$ (which, throughout this text is
assumed to be smooth), and given an incident field $u^{inc}$ for which
\begin{equation}
  \label{eq:Maxwell_inc}
  \Delta u^{inc}+k_1^2 u^{inc}=0 \qquad \rm{in}\ \mathbb{R}^2,
\end{equation}
the fields $u^1$ and $u^2$ can be obtained as the solutions to the
 Helmholtz equations
\begin{equation}
  \label{eq:Ac_e}
  \Delta u^1+k_1^2 u^1=0\qquad \rm{in}\ \Omega_1,
\end{equation}
\begin{equation}
  \label{eq:Ac_i}
  \Delta u^2+k_2^2 u^2=0 \qquad \rm{in}\ \Omega_2,
\end{equation}
which satisfy the transmission conditions
\begin{equation}
\label{eq:bc}
\begin{split}
 u^1 + u^{inc} &=u^2\qquad \rm{on}\ \Gamma \\
\frac{\partial u^1}{\partial n} + \frac{\partial u^{inc}}{\partial n}&=\nu \frac{\partial u^2}{\partial n}\qquad \rm{on}\ \Gamma
\end{split}
\end{equation}
as well as the Sommerfeld radiation condition at infinity:
\begin{equation}\label{eq:radiation}
\lim_{|r|\to\infty}r^{1/2}(\partial u^1/\partial r - ik_1u^1)=0.
\end{equation}
In these equations $n$ denotes the unit normal on $\Gamma$ that points
into $\Omega_2$ and $k_j$ denotes the wavenumber in the region $D_j$
($j=1$, $2$); for a given frequency $\omega$ we have
$k_i=\omega\sqrt{\mu_0\epsilon_j}$. The constant $\nu$
in~\eqref{eq:bc}, finally, is given by $\nu=1$ under $E$-polarization
and $\nu=\epsilon_1/\epsilon_2$ under $H$-polarization.

\begin{remark}
  As is well known, the transmission
  problem~\eqref{eq:Ac_e}-\eqref{eq:radiation} admits a unique
  solution provided the wavenumbers $k_j$, $j=1$, $2$ are real
  numbers~\cite{KressRoach,KleinmanMartin}.  In this text it is
  assumed that either $k_j$, $j=1$, $2$ are real numbers or,
  otherwise, that equations~\eqref{eq:Ac_e}-\eqref{eq:radiation} admit
  unique solutions.
\end{remark}

\section{Classical boundary integral formulations\label{di_ind_cfie}}
A variety of integral equation formulations for the transmission
problem~\eqref{eq:Ac_e}-\eqref{eq:radiation} exist
(e.g.~\cite{KressRoach,costabel-stephan,KleinmanMartin,muller}). In
the spirit of the general boundary integral method, these formulations
express the solution of equation~\eqref{eq:Ac_e}-\eqref{eq:radiation}
in terms of boundary integral representations that include free-space
Green functions and surface potentials as well as unknown functions
defined on the boundary $\Gamma$ (the so-called ``integral
densities''); the unknown densities, in turn, are determined as
solutions of corresponding integral equations.

The two-dimensional free-space Green function $G_k$ for a wavenumber
$k$ is given by
\begin{equation}
G_k(\mathbf
x)=\frac{i}{4}H_0^1(k|\mathbf x|).
\end{equation} 
Given a wavenumber $k$, the single and double layer potentials for the
wavenumber $k$ and with respective densities $\varphi$ and $\psi$ are
functions defined throughout $\mathbb{R}^2\setminus\Gamma$ by means of
the integral expressions
\[
S_k[\varphi](\mathbf{z})=\int_\Gamma G_k(\mathbf z - \mathbf
y)\varphi(\mathbf y)ds(\mathbf y),\quad
\mathbf{z}\in\mathbb{R}^2\setminus\Gamma\quad \mbox{(Single Layer
  Potential)}
\]
and
\[
D_k[\psi](\mathbf{z})=\int_\Gamma \frac{\partial
  G_k(\mathbf z - \mathbf y)}{\partial n(\mathbf y)}\psi(\mathbf
y)ds(\mathbf y),\ \mathbf{z}\in\mathbb{R}^2\setminus\Gamma\quad
\mbox{(Double Layer Potential)}.
\]
Clearly $S_k[\varphi]$ and $D_k[\psi]$ are solutions of the Helmholtz
equation $\Delta \psi + k^2\psi =0$ in $\mathbb{R}^2\setminus\Gamma$.


The Green identities are often utilized in derivations of integral
equation formulations; in the present context they can be used to
express the solutions $u_1$ and $u_2$ in terms of single and double
layer potentials on $\Gamma$:
\begin{equation}\label{Green1}
  u^1(\mathbf z)= D_{k_1}[u^1](\mathbf{z})-S_{k_1}\left[\frac{\partial u^1}{\partial n}\right](\mathbf{z}),\ \mathbf{z}\in \Omega_1
\end{equation} 
and
\begin{equation}\label{Green2}
  u^2(\mathbf z) = -D_{k_2}[u^2](\mathbf{z})+S_{k_2}\left[\frac{\partial u^2}{\partial n}\right](\mathbf{z}),\ \mathbf{z}\in \Omega_2.
\end{equation}
Equations for the unknown boundary values of $u^j$ and $\frac{\partial
  u^j}{\partial n}$ on $\Gamma$ are obtained by enforcing the
transmission conditions~\eqref{eq:bc}; once these unknown functions
have been obtained the representation formulae~\eqref{Green1}
and~\eqref{Green2} can be used to evaluate the solutions $u^1$ and
$u^2$ throughout $\Omega_1$ and $\Omega_2$. In view of the use of Green's
formula in conjunction with the physical quantities $u^j$ and
$\frac{\partial u^j}{\partial n}$ as integral densities, the integral
equations obtained in this manner are called ``Direct''. Other
integral equations, such as the regularized integral equations
introduced in this paper, are ``Indirect'': the corresponding unknown
densities are not directly identifiable physical quantities.

In addition to Green's identities~\eqref{Green1} and~\eqref{Green2},
exterior and interior boundary values of single and double layer
potentials together with exterior and interior boundary values of the
normal derivatives of these potentials are needed in the derivation of
integral equations for transmission problems. A number of boundary
integral operators thus arise for a given wavenumber $k$, including
the {\em single layer} operator
\begin{equation}\label{eq:sl}
  S_k[\varphi](\mathbf x)=\int_\Gamma G_k(\mathbf x -\mathbf y)\varphi(\mathbf y)ds(\mathbf y),\ \mathbf{x}\ {\rm on} \ \Gamma ,
\end{equation} 
the {\em double layer} operator $D_k$ and its {\em adjoint} $D^*_k$
\begin{equation}
\label{eq:double}
D_k[\varphi](\mathbf x)=\int_{\Gamma}\frac{\partial G_k(\mathbf x-\mathbf y)}{\partial n(\mathbf y)}\varphi(\mathbf y)ds(\mathbf y),\ \mathbf x\ {\rm on}\ \Gamma
\end{equation}
and 
\begin{equation}
\label{eq:adj_double}
D_k^*[\varphi](\mathbf x)=\int_{\Gamma}\frac{\partial G_k(\mathbf x-\mathbf y)}{\partial n(\mathbf x)}\varphi(\mathbf y)ds(\mathbf y),\ \mathbf x\ {\rm on}\ \Gamma,
\end{equation}
and the {\em hypersingular} operator
\begin{equation}\label{eq:normal_double}
  \begin{split}
N_k [\varphi](\mathbf x) &= {\rm PV} \int_\Gamma \frac{\partial^{2}G_k(\mathbf x -\mathbf y)}{\partial n(\mathbf x) \partial n(\mathbf y)} \varphi(\mathbf y)ds(\mathbf y)\\
&=k^{2}\int_\Gamma G_k(\mathbf x -\mathbf y)
(n(\mathbf x)\cdot n(\mathbf y))\varphi(\mathbf y)ds(\mathbf y)+ {\rm PV}
\int_\Gamma \partial_s G_k(\mathbf x -\mathbf y)\partial_s \varphi(\mathbf y)ds(\mathbf y).
\end{split}
\end{equation}
The operator $N_k$ equals the limit of the normal derivative of the
double layer potential on $\Gamma$ from the outside of the domain.
Both expressions in~\eqref{eq:normal_double} contain Cauchy Principal
Value integrals; the second line in this equation, further, utilizes
the tangential derivative operator $\partial_s$ on the curve $\Gamma$.

The mapping properties of the boundary integral operators above in
appropriate Sobolev spaces $H^{s}(\Gamma)$ of functions defined on the
curve $\Gamma$~\cite{Kress,Saranen}) ($s\in \mathbb{R}$) play central
roles in determination of existence and uniqueness of solutions for
the integral equations of scattering theory. The following theorem
outlines relevant results in these regards. (In what follows and
throughout this paper, for a complex number $z$ the symbols $\Re(z)$
and $\Im(z)$ denote the real and imaginary parts of $z$.)
\begin{theorem}\label{regL} 
  Let $k$ be a complex number satisfying $\Re{k}\geq 0$ and
  $\Im(k)\geq 0$ and assume $\Gamma$ is a smooth curve. Then for all
  $s\geq 0$ the boundary integral
  operators~\eqref{eq:sl}--\eqref{eq:normal_double} define continuous
  mappings $S_k:H^{s-1/2}(\Gamma)\to H^{s+1/2}(\Gamma)$;
  $N_k:H^{s+1/2}(\Gamma)\to H^{s-1/2}(\Gamma)$;
  $D_k^*:H^{s+1/2}(\Gamma)\to H^{s+7/2}(\Gamma)$; and
  $D_k:H^{s+1/2}(\Gamma)\to H^{s+7/2}(\Gamma)$ for all $s\geq 0$.
\end{theorem}
\begin{proof}
  The first two properties are classical~\cite{Saranen}. The last two
  properties result from the increased regularity of the corresponding
  kernels. Indeed, assuming that the boundary curve $\Gamma$ is given
  by $\Gamma=\{x(t)=(x_1(t),x_2(t)),\quad 0\leq t\leq 2\pi\}$ where
  $x:\mathbb{R}\to \mathbb{R}^2$ is smooth and $2\pi$ periodic, the
  operators $D_k^*$ can be expressed~\cite{Kress} in the parametric form
\begin{equation}\label{dstarparam}
(D_k^*\varphi)(x(t))=\frac{1}{|x'(t)|}\int_0^{2\pi}L(t,\tau)\varphi(x(\tau))d\tau
\end{equation}
where the kernel $L(t,\tau)$ is given by
$$L(t,\tau)=\frac{k}{4\pi}\frac{J_1(k|x(t)-x(\tau)|)}{|x(t)-x(\tau)|}\ \frac{(x(t)-x(\tau))\cdot n(t)}{|x(t)-x(\tau)|^2}\ |x(t)-x(\tau)|^2\log\left(4\sin^2{\frac{t-\tau}{2}}\right)+L_1(t,\tau)$$
where $L_1$ is a smooth function of the variables $t$ and $\tau$ for
all $(t,\tau)\in[0,2\pi]\times[0,2\pi]$, and where $J_1$ denotes the
first order Bessel function of the first kind. It is easy to check
that for smooth curves $\Gamma$ the kernel $L$ is a $C^1$ function of
the variables $t$ and $\tau$ for $(t,\tau)\in[0,2\pi]\times[0,2\pi]$
whose second derivative has a logarithmic singularity. Thus, $D_k^*$
is regularizing of three orders, that is $D_k^*:H^{s+1/2}(\Gamma)\to
H^{s+7/2}(\Gamma)$. Similar considerations lead to the result about
the mapping properties of the operator $D_k$.
\end{proof}

In addition to the mapping properties laid down in Theorem~\ref{regL},
below in this text we also utilize additional smoothing properties
presented in Theorem~\ref{thmrev2} that arise as certain differences
of boundary integral operators are considered.
\begin{theorem}\label{thmrev2}
  Let two wavenumbers $\kappa_1$ and $\kappa_2$ such that
  $\Re{\kappa_j}\geq 0$ and $\Im(\kappa_j)\geq 0$ be given. Then, the
  operators $S_{\kappa_1}-S_{\kappa_2}$ and
  $N_{\kappa_1}-N_{\kappa_2}$ are regularizing operators by three
  orders and by one order, respectively. More precisely, for all
  $s\geq 0$ these operator differences define continuous mappings
  $S_{\kappa_1}-S_{\kappa_2}:H^{s-1/2}(\Gamma)\to H^{s+5/2}(\Gamma)$
  and $N_{\kappa_1}-N_{\kappa_2}:H^{s+1/2}(\Gamma)\to
  H^{s+3/2}(\Gamma)$.
\end{theorem}
\begin{proof}
  Lemma~10 in~\cite{lint-br} provides a proof for the single layer
  case in a slightly different setting. The corresponding result for
  the hypersingular operator may be established in a similar manner
  (the difference $N_{\kappa_1}-N_{\kappa_2}$ can be expressed as an
  operator with a weakly singular kernel) or, alternatively, from the
  result for $S_{\kappa_1}-S_{\kappa_2}$ together with the relation
\[
(N_{\kappa_1} - N_{\kappa_2})S_{\kappa_1} = (D_{\kappa_1}^*)^2 -N_{\kappa_2}(S_{\kappa_1}-S_{\kappa_2})-(D_{\kappa_2}^*)^2
\]
which follows directly from the Calder\'on relation~\cite{Nedelec}.
\end{proof}
\begin{remark}
  For improved readability in what follows we rely on a slight
  notational abuse: the boundary integral operators associated with
  the wavenumbers $k_1$ and $k_2$ of regions $\Omega_1$ and $\Omega_2$ will be
  denoted by means of a subindex 1 or 2, respectively, i.e., $S_1 =
  S_{k_1}$, $S_2 = S_{k_2}$ and, further, $D_j=D_{k_j}$,
  $D_j^*=D_{k_j}^*$ and $N_j=N_{k_j}$ for $j=1,2$. This notation is
  reserved for the wavenumbers $k_1$ and $k_2$ exclusively, however:
  for other wavenumbers $k = \kappa_1$ and $k = \kappa_1$ to be
  considered the full wavenumber dependence will be displayed as part
  of the notation of the boundary integral operators---e.g. the single
  layer operators with wavenumber $k = \kappa_1$ will be denoted by
  $S_{\kappa_1}$, etc; the notation $S_1$ is reserved for the
  operator $S_{k_1}$.
\end{remark}

As indicated earlier in this section, utilizing the Green
identities~\eqref{Green1} and~\eqref{Green2} to represent the fields
in the regions $\Omega_1$ and $\Omega_2$, respectively, integral
equations for the problem~\eqref{eq:Ac_e}-\eqref{eq:radiation} are
obtained as the transmission conditions~\eqref{eq:bc} are enforced on
the corresponding functions $u^1$ and $u^2$
(cf.~\cite{KittapaKleinman}). Such equations can be obtained by
combining the Green's identities~\eqref{Green1} and~\eqref{Green2}
together with corresponding Green's formulae for the incident fields
(which in view of the assumption~\eqref{eq:Maxwell_inc} satisfy the
Helmholtz equation throughout space), and subsequently evaluating the
interior and exterior boundary values on $\Gamma$ of the fields
$u^j,j=1,2$ and their normal derivatives. Two pairs of integral
equations thus result, one pair arising as limits are taken from the
interior of $\Gamma$, and the other as limits are taken from the
exterior of $\Gamma$. In view of the transmission
conditions~\eqref{eq:bc} these equations may be expressed as integrals
involving the values of the total exterior field $u = u^1+u^{inc}$ and
its normal derivative $\frac{\partial u}{\partial n}$ on $\Gamma$. A
linear combination of the inside and outside integral equations that
express the total field as an integral, together with another linear
combination of the integral equations that express the normal
derivative of the total field as an integral, can then be laid down,
in a search for a well posed, uniquely solvable Fredholm system of two
equations and two unknowns for the transmission problem at
hand. Particular selections of the values of the coefficients of these
linear combinations, for example, yield the following two pairs of
integral equations
\begin{eqnarray}\label{eq:system_trans1}
  u(\mathbf x) + (D_2-D_1)[u](\mathbf x) -(\nu^{-1} S_2 - S_1)\left[\frac{\partial u}{\partial n}\right](\mathbf x)&=&u^{inc}(\mathbf x)\nonumber\\
  \frac{\nu^{-1}+1}{2}\left[\frac{\partial u}{\partial n}\right](\mathbf x) + (D_1^*-\nu^{-1} D^*_2)\left[\frac{\partial u}{\partial n}\right](\mathbf x) -(N_1 - N_2)\left[u\right](\mathbf x)&=&\frac{\partial u^{inc}}{\partial n}(\mathbf x), 
\end{eqnarray}
and
\begin{eqnarray}\label{eq:system_trans}
  \frac{\nu^{-1}+1}{2}u(\mathbf x) + (D_2-\nu^{-1}D_1)[u](\mathbf x) +\nu^{-1} (S_1 - S_2)\left[\frac{\partial u}{\partial n}\right](\mathbf x)&=&\nu^{-1} u^{inc}(\mathbf x)\nonumber\\
  \frac{\nu^{-1}+1}{2}\frac{\partial u}{\partial n}(\mathbf x) + (D_1^*-\nu^{-1} D^*_2)\left[\frac{\partial u}{\partial n}\right](\mathbf x) -(N_1 - N_2)[u](\mathbf x)&=&\frac{\partial u^{inc}}{\partial n}(\mathbf x),
\end{eqnarray}
($\mathbf x\in\Gamma$), each one of which amounts to a well posed
formulation of the transmission
problem~\eqref{eq:Ac_e}-\eqref{eq:radiation}.  These systems of
equations are presented in~\cite{KittapaKleinman}
and~\cite{KleinmanMartin}, respectively. Note that in both of these
systems of equations the hypersingular operators $N_j,j=1,2$ appear in
the form $N_1-N_2$---which can be expressed in terms of a weakly
singular kernel. Thus, given the mapping properties recounted in
Theorems~\ref{regL} and~\ref{thmrev2}, the boundary integral operators
in equations~\eqref{eq:system_trans1} and~\eqref{eq:system_trans} are
either compact operators from $H^p(\Gamma)$ to $H^p(\Gamma)$ ($p\geq
\frac{1}{2}$), or appear as compact differences of operators between
these spaces---and thus equations~\eqref{eq:system_trans1}
and~\eqref{eq:system_trans} are integral equations of the second
kind. The unique solvability of equations~\eqref{eq:system_trans1}
and~\eqref{eq:system_trans} is established in~\cite{KleinmanMartin}.

While the formulations~\eqref{eq:system_trans1}
and~\eqref{eq:system_trans} are identical to each other in the case
$\nu=1$, their properties differ significantly for $\nu\neq 1$. For
instance, we have observed (see Section~\ref{prec}) that for some
cases including high-contrast dielectrics the use of the
formulation~\eqref{eq:system_trans} gives rise to important reductions
in the numbers of GMRES iterations required to reach a given
tolerance, with corresponding reductions in computational times,
vis-a-vis those required by~\eqref{eq:system_trans1}. We attribute
this phenomenon to the fact that the spectrum of the two-by-two matrix
of integral operators on the left-hand-side of~\eqref{eq:system_trans}
accumulates only at the point $\frac{\nu^{-1}+1}{2}$, while the
corresponding spectrum of the matrix operator on the
left-hand-side~\eqref{eq:system_trans1} accumulates at two points,
namely $1$ and $\frac{\nu^{-1}+1}{2}$. The spectrum of the operator
in~\eqref{eq:system_trans1} is therefore likely to be more widely
spread in the complex plane than that of~\eqref{eq:system_trans}, thus
conceivably giving rise to the observed differences in the number of
iterations required by the GMRES solver to reach a given tolerance for
equations~\eqref{eq:system_trans1} and~\eqref{eq:system_trans}.


Yet another linear combination of the pairs of integral equations
arising from exterior and interior limits can be used to effect the
cancellation of the identity terms. This procedure results in the
following first kind integral equations with a positive definite
principal part
\begin{equation}\label{eq:system_trans_FK}
\begin{split}
  -(D_1+D_2)[u](\mathbf x) + (\nu^{-1} S_2 + S_1)\left[\frac{\partial u}{\partial n}\right](\mathbf x)&= u^{inc}(\mathbf x)\\
  -(N_1 +\nu N_2)[u](\mathbf x)+(D^*_1+D^*_2)\left[\frac{\partial
      u}{\partial n}\right](\mathbf x)&=\frac{\partial
    u^{inc}}{\partial n}(\mathbf x),
\end{split}
\end{equation}
($\mathbf x\in\Gamma$).
The unique solvability of the first-kind
system~\eqref{eq:system_trans_FK} $H^{s+1/2}(\Gamma)\times
H^{s-1/2}(\Gamma),\ s\geq 0$ was established
in~\cite{costabel-stephan}.

Indirect versions of the direct second-kind and first-kind equations
discussed above in this section can be obtained by expressing the
fields $u^j,j=1,2$ as linear combinations of single and double layer
potentials and manipulating the resulting interior and exterior
boundary integral equations to parallel the characteristics obtained
for the direct equations discussed above.  The numerical properties of
each resulting indirect integral equation is essentially identical to
that associated with the corresponding direct-equation counterpart.

In the following section we introduce a new methodology for derivation
of well posed {\em regularized} transmission integral equations;
similar ideas were used recently~\cite{bept} in the treatment of the
problem of scattering by perfectly conducting scatterers.  As
demonstrated in Section~\ref{prec}, the new transmission integral
equations enable solution of transmission problems at reduced numbers of GMRES iterations, and correspondingly reduced
computational times, over those resulting from all of the formulations
discussed above in this section.

\section{Regularized integral equations for transmission problems}

\subsection{Regularizing operators I: layer-potential regularization\label{reg-I}}
In this section we derive {\em regularized} integral equations for
transmission problems. To do this we consider acoustic fields $u^1$
and $u^2$ given by
\begin{equation}\label{ansatz1}
  u^1(\mathbf{z})= D_1\left[R_{11}[a]+R_{12}[b]\right](\mathbf{z})-S_1\left[R_{21}[a]+R_{22} [b]\right](\mathbf{z}),\quad \mathbf{z}\in\mathbb{R}^2\setminus\Gamma
\end{equation}
and
\begin{equation}\label{ansatz2}
  u^2(\mathbf{z})= -D_2\left[R_{11} [a]+R_{12} [b]-a\right](\mathbf{z})+\nu^{-1}S_2\left[R_{21} [a]+R_{22} [b]-b\right](\mathbf{z}),\quad \mathbf{z}\in\mathbb{R}^2\setminus\Gamma,
\end{equation}
where $R_{ij}$ ($i=1,2$, $j=1,2$) are regularizing operators (acting
on spaces of functions defined on $\Gamma$) which are to be selected
in what follows. In view of the jump conditions of the various
boundary layer potentials, the fields $u^1$ and $u^2$ solve the
transmission problem~\eqref{eq:Ac_e}--\eqref{eq:radiation} if and only
if the system of integral equations
\begin{equation}\label{eq:reg_diel}
\begin{split}
  &\left(\frac{I}{2}-D_2+(D_1+D_2)R_{11}-(S_1+\nu^{-1}S_2)R_{21}\right)[a]\\
  &+\left(\nu^{-1}S_2 + (D_1+D_2)R_{12}-(S_1+\nu^{-1}S_2)R_{22}\right)[b]=-u^{inc}|_\Gamma\\
  \\
  &\left(-\nu N_2 +(N_1+\nu N_2)R_{11}-(D^*_1+D^*_2)R_{21}\right)[a]\\
  +&\left(\frac{I}{2}+K^T_2+(N_1+\nu
    N_2)R_{12}-(D^*_1+D^*_2)R_{22}\right) [b]=-(\partial
  u^{inc}/\partial n)|_\Gamma.
\end{split}
\end{equation}
is satisfied.  The system of equations~\eqref{eq:reg_diel} can be
expressed in the form
\begin{equation}\label{mat-eqn}
\mathcal{A}\left(\begin{array}{c} a\\ b\end{array}\right)=-\left(\begin{array}{c}u^{inc}|_\Gamma\\(\partial u^{inc}/\partial n)|_\Gamma\end{array}\right)
\end{equation}
in terms of a certain $2\times 2$ matrix $\mathcal{A}$ with operator
entries.

Our goal is to produce matrix operators $\mathcal{R}=(R_{ij})_{1\leq
  i,j\leq 2}$ such that the matrix operators $\mathcal{A}$ in
equation~\eqref{mat-eqn} is (i)~A compact perturbation of the a
multiple of the identity matrix operator, and (ii)~It is
invertible. In addition, we strive to construct operators
$\mathcal{R}$ that lend themselves to (iii)~Accurate and efficient
numerical implementations, and require iv)~Reduced numbers of
iterations when the system of equations is solved by means of the
iterative linear algebra solver GMRES.

Letting $\mathcal{A}_{ij},\ 1\leq i,j\leq 2$ denote the entries of the
matrix $\mathcal{A}$, we consider at first the operator
$\mathcal{A}_{21}=-\nu N_2 +(N_1+\nu N_2)R_{11}-(D^*_1+D^*_2)R_{21}$,
and we seek to obtain operators $R_{11}$ and $R_{21}$ such that
$\mathcal{A}_{21}$ maps $H^s(\Gamma)$ into $H^{s+1}(\Gamma)$. Given
the smoothing properties of the operators $D_j^*$ ($j=1,2$), we seek
at first to obtain an operator $R_{11}$ for which $-\nu N_2 +(N_1+\nu
N_2)R_{11}$ maps $H^s(\Gamma)$ into $H^{s+1}(\Gamma)$. It is natural
to seek $R_{11}$ of the form $R_{11}=\alpha I$, so that $-\nu N_2
+(N_1+\nu N_2)R_{11}=\alpha N_1-\nu(1-\alpha)N_2$. The latter operator
has the desired property provided that $\alpha=\nu(1-\alpha)$, that is
$\alpha=\frac{\nu}{1+\nu}$. We then seek an operator $R_{21}$ for
which
$\mathcal{A}_{11}=\frac{I}{2}-D_2+(D_1+D_2)R_{11}-(S_1+\nu^{-1}S_2)R_{21}$
is a compact perturbation of the identity operator. Given the
selection made above for the operator $R_{11}$ and the smoothing
properties of the operators $D_j^*$, we need to obtain an operator
$\mathcal{R}_{21}$ such that $(S_1+\nu^{-1}S_2)R_{21}=-\frac{I}{2}+R$
for some continuous operator $R:H^s(\Gamma)\to H^{s+1}(\Gamma)$. It is
natural to seek for operators of the form $R_{21}=\beta N_{\kappa_1},\
\kappa_1\in\mathbb{C}$. (We emphasize here that use of a wavenumber
$\kappa_1$ which is different from that inherent in the given
transmission problem and that, indeed, may be taken to be {\em
  complex}, plays a crucial role to insure invertibility of the
resulting systems of equations; see Section~\ref{invertibility}.) We
have that $\beta(S_1+\nu^{-1}S_2)N_{\kappa_1}=\beta
S_1(N_{\kappa_1}-N_1)+\beta\nu^{-1}S_2(N_{\kappa_1}-N_2)-\frac{\beta(1+\nu^{-1})}{4}I+\beta
D_1^2 +\beta\nu^{-1}D_2^2$, and thus we see that
$\beta=\frac{2\nu}{1+\nu}$ would render operators with the desired
property. We require next that the operator
$\mathcal{A}_{22}=\frac{I}{2}+D^*_2+(N_1+\nu
N_2)R_{12}-(D^*_1+D^*_2)R_{22}$ is a compact perturbation of the
identity operator. Similar considerations lead to the choice
$R_{12}=\gamma S_{\kappa_1}$ where $\gamma=-\frac{2}{1+\nu}$. Given
that the choices presented above lead to $R_{21}=2N_{\kappa_1}R_{11}$,
it is natural to choose $R_{22}=\frac{1}{1+\nu}I$.

In conclusion, the matrix operator $\mathcal{R}$ takes on the form
\begin{eqnarray}\label{eq:tildR1}
R_{11}&=&\frac{\nu}{1+\nu}I\nonumber\\
R_{12}&=&-\frac{2}{1+\nu}S_{\kappa_1}\nonumber\\
R_{21}&=&\frac{2\nu}{1+\nu}N_{\kappa_1}\nonumber\\
R_{22}&=&\frac{1}{1+\nu}I,
\end{eqnarray}
where $\kappa_1$ is a wavenumber such that $\Re(\kappa_1)\geq 0$ and
$\Im(\kappa_1)\geq 0$. The following theorem establishes that the
matrix operator that enters the integral equation
formulation~\eqref{eq:reg_diel} with regularizing operator
$\mathcal{R}$ given in equation~\eqref{eq:tildR1} is a compact
perturbation of the identity matrix.
\begin{theorem}\label{thm2}
  Let $\mathcal{A}$ be the operator in the left-hand-side of
  equation~\eqref{eq:reg_diel}, where $\mathcal{R}$ is given by
  equation~\eqref{eq:tildR1}. Then the operator $\mathcal{A}$ is a
  compact perturbation of the identity matrix in the space
  $H^{s}(\Gamma)\times H^{s}(\Gamma)$ for all $s\geq\frac{1}{2}$.
\end{theorem}
\begin{proof} For the $1,1$ component of the matrix operator
  $\mathcal{A}$, which is given by
\begin{eqnarray}
\mathcal{A}_{11}&=&\frac{I}{2}-\frac{2\nu}{1+\nu}(S_1+\nu^{-1}S_2)N_{\kappa_1}-D_2+\frac{\nu}{1+\nu}(D_1+D_2)\nonumber\\
&=&\frac{I}{2}-2S_{\kappa_1} N_{\kappa_1}-\frac{2\nu}{1+\nu}[(S_1-S_{\kappa_1})+\nu^{-1}(S_2-S_{\kappa_1})]N_{\kappa_1}\nonumber\\
&-&D_2+\frac{\nu}{1+\nu}(D_1+D_2),
\end{eqnarray}
we have $\mathcal{A}_{11}=I+\mathcal{A}_{11}^r$ where
\begin{eqnarray}
\mathcal{A}_{11}^r&=&-2(D_{\kappa_1}^*)^2-\frac{2\nu}{1+\nu}[(S_1-S_{\kappa_1})+\nu^{-1}(S_2-S_{\kappa_1})]N_{\kappa_1}\nonumber\\
&-&D_2+\frac{\nu}{1+\nu}(D_1+D_2).\nonumber
\end{eqnarray}
Given that $D_j:H^{s}(\Gamma)\to H^{s+3}(\Gamma)$,
$D_{\kappa_1}^*:H^{s}(\Gamma)\to H^{s+3}(\Gamma)$,
$N_{\kappa_1}:H^{s}(\Gamma)\to H^{s-1}(\Gamma)$,
$S_j-S_{\kappa_1}:H^{s-1}(\Gamma)\to H^{s+2}(\Gamma)$ for $j=1,2$, we
see that $\mathcal{A}_{11}^r:H^{s}(\Gamma)\to H^{s+2}(\Gamma)$.

The $1,2$ entry of the matrix operator $\mathcal{A}$ is given by
\begin{eqnarray}
  \mathcal{A}_{12}&=&\nu^{-1}S_2-\frac{2}{1+\nu}(D_1+D_2)S_{\kappa_1}-\frac{1}{1+\nu}(S_1+\nu^{-1}S_2)\nonumber\\
  &=&\frac{1}{1+\nu}(S_2-S_1)-\frac{2}{1+\nu}(D_1+D_2)S_{\kappa_1};
\end{eqnarray}
which, as it is checked easily, maps $H^{s}(\Gamma)$ into
$H^{s+1}(\Gamma)$.

The $2,1$ entry of $\mathcal{A}$ is given by
\begin{eqnarray}
\mathcal{A}_{21}&=&-\nu N_2 +\frac{\nu}{1+\nu}(N_1+\nu N_2)-\frac{2\nu}{1+\nu}(D_1^*+D_2^*)N_{\kappa_1}\nonumber\\
&=&\frac{\nu}{1+\nu}(N_1-N_2)-\frac{2\nu}{1+\nu}(D_1^*+D_2^*-2D_{\kappa_1}^*)N_{\kappa_1}\nonumber\\
&-&\frac{4\nu}{1+\nu}N_{\kappa_1}D_{\kappa_1}.
\end{eqnarray}
Given that $N_1-N_2:H^{s}(\Gamma)\to H^{s+1}(\Gamma)$, $D_j^*-D_{\kappa_1}^*:H^{s}(\Gamma)\to H^{s+3}(\Gamma)$, $N_{\kappa_1}:H^{s}(\Gamma)\to H^{s-1}(\Gamma)$, $D_{\kappa_1}:H^{s}(\Gamma)\to H^{s+3}(\Gamma)$ for $j=1,2$, we see that $\mathcal{A}_{21}:H^{s}(\Gamma)\to H^{s+1}(\Gamma)$. 

Finally, the $2,2$ entry of $\mathcal{A}$, which is given by
\begin{eqnarray}
  \mathcal{A}_{22}&=&\frac{I}{2}-\frac{2}{1+\nu}(N_1+\nu N_2)S_{\kappa_1} +D_2^*-\frac{1}{1+\nu}(D_1^*+D_2^*)\nonumber\\
  &=&\frac{I}{2}-2N_{\kappa_1} S_{\kappa_1} -\frac{2}{1+\nu}[(N_1-N_{\kappa_1})+\nu (N_2-N_{\kappa_1})]S_{\kappa_1} \nonumber\\
  &+&D_2^*-\frac{1}{1+\nu}(D_1^*+D_2^*)\nonumber
\end{eqnarray}
can be expressed in the form $\mathcal{A}_{22} = I +
\mathcal{A}_{22}^r$, where
\begin{eqnarray}
  \mathcal{A}_{22}^r&=&-2D_{\kappa_1}^2-\frac{2}{1+\nu}[(N_1-N_{\kappa_1})+\nu (N_2-N_{\kappa_1})]S_{\kappa_1} +D_2^*-\frac{1}{1+\nu}(D_1^*+D_2^*).
\end{eqnarray}
Given that $N_j-N_{\kappa_1}:H^{s}(\Gamma)\to H^{s+1}(\Gamma)$, the
mapping properties of the single and double layer potentials tell us
that $\mathcal{A}_{22}^r$ maps $H^{s}(\Gamma)$ continuously into
$H^{s+2}(\Gamma)$, and the proof is complete.
\end{proof}

\subsection{Invertibility of the integral formulation~\eqref{eq:reg_diel}
  with $\mathcal{R}$ given by~\eqref{eq:tildR1}\label{invertibility}}
Having established that the operator $\mathcal{A}$ equals the sum of
the identity operator and a compact operator we turn next to its
invertibility, or what is equivalent (in view of the Fredholm theory),
to its injectivity. The freedom to select {\it complex} values of the
regularization wavenumber $\kappa_1$ plays a fundamental role in these
regards.
\begin{theorem}\label{thm3}
  Let $R_{11}$ and $R_{22}$ be given as in equation~\eqref{eq:tildR1},
  and let $R_{12}$ and $R_{21}$ satisfy the positivity relations
\begin{equation}\label{positive1}
 - \Im \int_\Gamma (R_{12} [b])\ \overline{b}\ ds > 0\ \mbox{for all}\ b\in H^{-1/2}(\Gamma)\ b\ne 0
\end{equation}
\begin{equation}\label{positive2}
  \Im \int_\Gamma (R_{21} [a])\ \overline{a}\ ds > 0\ \mbox{for all}\ a\in H^{1/2}(\Gamma)\ a\ne 0.
\end{equation}
Then the operator $\mathcal{A}:H^{s}(\Gamma)\times H^{s}(\Gamma)\to
H^{s}(\Gamma)\times H^{s}(\Gamma)$ is injective.
\end{theorem}
\begin{proof}
  Let $(a,b)$ be a solution of equations~\eqref{eq:reg_diel} with zero
  right-hand side and let $u^1$ and $u^2$ be given by
  equations~\eqref{ansatz1} and~\eqref{ansatz2}.  In view of classical
  uniqueness results for transmission problems we see that $u_1=0$ in
  $ \Omega_1$ and $u_2=0$ in $\Omega_2$. It follows that
\begin{equation}\label{zero_lim}
  u_+^1=\partial u_+^1/\partial n=0\quad \mbox{and}\quad u_-^2=\partial
  u_-^2/\partial n=0\quad \mbox{on}\quad \Gamma, 
\end{equation}
where the subscript $+$ (resp. the subscript $-$) denotes boundary
values that are obtained as $\Gamma$ is approached from the exterior
domain $\Omega_1$ (resp. from the interior domain $\Omega_2$).  Noting
that equations~\eqref{ansatz1} and~\eqref{ansatz2} define $u^1$ and
$u^2$ everywhere outside $\Gamma$, we may also evaluate the limits of
$u^1$ (resp. $u^2$) as $\Gamma$ is approached from the interior domain
$\Omega_2$ (resp. the exterior domain $\Omega_1$). Using the
relations~\eqref{zero_lim} together with the well known jump formulas
for the various layer potentials we thus obtain
\begin{eqnarray}\label{jump_final}
  u_{-}^1=-R_{11}[a]-R_{12} [b],&&\partial u_{-}^1/\partial n=-R_{21} [a]-R_{22} [b]\nonumber\\
  u_{+}^2=a-R_{11} [a]-R_{12} [b],&&\partial u_{+}^2/\partial n =\nu^{-1}(b-R_{21}[a]-R_{22} [b]),
\end{eqnarray}
and, therefore,
\begin{eqnarray}
\nu \int_\Gamma u_{+}^2\ \overline{\partial u_+^2/\partial n}\ ds&=&\int_\Gamma a\ \overline{b}\ ds -\int_\Gamma a\ \overline{R_{21}[a]}\ ds - \int_\Gamma a\ \overline{R_{22}[b]}\ ds\nonumber\\
&-&\int_\Gamma (R_{11}[a])\ \overline{b}\ ds - \int_\Gamma (R_{12}[b])\ \overline{b}\ ds + \int_\Gamma u_{-}^1\ \overline{\partial u_{-}^1/\partial n}\ ds.\nonumber
\end{eqnarray}
Using the definition of the operators $R_{11}$ and $R_{22}$ given in equation~\eqref{eq:tildR1} we then obtain
\begin{eqnarray}
\nu \int_\Gamma u_{+}^2\ \overline{\partial u_+^2/\partial n}\ ds&=&-\int_\Gamma a\ \overline{R_{21}[a]}\ ds - \int_\Gamma (R_{12}[b])\ \overline{b}\ ds\nonumber\\
&+&\int_\Gamma u_{-}^1\ \overline{\partial u_{-}^1/\partial n}\ ds\nonumber
\end{eqnarray}
which in view of the Green identity can be also expressed in the form
\begin{eqnarray}
\nu \int_\Gamma u_{+}^2\ \overline{\partial u_+^2/\partial n}\ ds&=&-\int_\Gamma a\ \overline{R_{21}[a]}\ ds - \int_\Gamma (R_{12}[b])\ \overline{b}\ ds\nonumber\\
&+&\int_{\Omega_2}(-k_1^2|u^1|^2+|\nabla u^1|^2)dx.\nonumber
\end{eqnarray}
Taking the imaginary part of this equation we obtain
\begin{equation}\label{eq:uniq1}
\nu\ \Im \int_\Gamma u_{+}^2\ \overline{\partial u_+^2/\partial n}\ ds=\Im \int_\Gamma (R_{21}[a])\ \overline{a}\ ds -\Im \int_\Gamma (R_{12}[b])\ \overline{b}\ ds
\end{equation}
and, thus, in view of equations~\eqref{positive1}
and~\eqref{positive2}
\begin{equation}\label{eq:final}
\Im \int_\Gamma u_{+}^2\ \overline{\partial u_+^2/\partial n}\ ds\geq 0.
\end{equation} 
Since $u^2$ is a radiative solution of the Helmholtz equation with
wavenumber $k_2$ in the domain $\Omega_1$ it
follows~\cite[p. 78]{KressColton} that $u^2=0$ in $\Omega_1$, and,
thus, $\Im \int_\Gamma (R_{21} [a])\ \overline{a}\ ds=0$ and $\Im
\int_\Gamma (R_{12} b)\ \overline{b}\ ds=0$. Invoking once again
equations~\eqref{positive1} and~\eqref{positive2}, finally, we obtain
$a=b=0$ on $\Gamma$. It follows that the operator $\mathcal{A}$ is
injective, as claimed.
\end{proof}

\begin{theorem}

 Let $\mathcal{R}$ be defined as in equations~\eqref{eq:tildR1} and let
  $\kappa_1$ be given by
\begin{equation}\label{kappa1}
  \kappa_1=\kappa+i\varepsilon,\qquad
  \mbox{with}\quad\kappa>0\quad\mbox{and}\quad \varepsilon>0.
\end{equation}
Then the operator $\mathcal{A}:H^{s}(\Gamma)\times H^{s}(\Gamma)\to
H^{s}(\Gamma)\times H^{s}(\Gamma)$ is invertible for all
$s\geq\frac{1}{2}$.
\end{theorem}
\begin{proof}
  Given the definitions~\eqref{eq:tildR1} of the operators $R_{12}$
  and $R_{21}$ and the results in Appendix~\ref{app1}
  (Lemma~\ref{sixp1} and Lemma~\ref{sixp2}), we see that the operators
  $R_{12}$ and $R_{21}$ satisfy the positivity
  relations~\eqref{positive1} and~\eqref{positive2} respectively. It
  therefore follows from Theorem~\ref{thm3} that the operator
  $\mathcal{A}$ is injective. In view of Theorem~\ref{thm2} the
  claimed invertibility follows from the Fredholm theory and the
  proof is thus complete.
\end{proof}

\subsection{Regularizing operators II: Principal symbol
  regularization\label{reg-II}} In this section we introduce an
alternative choice of regularizing operators $\mathcal{R}$ that
results as the operators in equation~\eqref{eq:tildR1} are replaced by
operators that, for a given parametrization of the underlying curve
$\Gamma$, have the same leading order asymptotics in Fourier space. In
what follows we assume, without loss of generality, that $\Gamma$ is
represented by means of a smooth and $2\pi$ periodic parametrization
$x:[0,2\pi]\to\Gamma$.

For a given $\kappa>0$ and $\varepsilon>0$ let us define the following
functions
\begin{equation}\label{eq:PS}
  p^N(\xi;\kappa+i\varepsilon)=-\frac{1}{2}\sqrt{|\xi|^2-(\kappa+i\varepsilon)^2}\qquad p^S(\xi;\kappa+i\varepsilon)=\frac{1}{2\sqrt{|\xi|^2-(\kappa+i\varepsilon)^2}}
\end{equation}
for real values of the argument $\xi$. The square root featured in
equation~\eqref{eq:PS} is defined so that
$\Im(p^N(\xi;\kappa+i\varepsilon)) >0$ and
$\Im(p^S(\xi;\kappa+i\varepsilon))>0$ for all $\xi$. Using these
functions together with the given parametrization $x =x(t)$ of the
curve $\Gamma$ we define (cf.~\cite{Saranen}) operators
$\sigma_{\kappa+i\varepsilon}^{N,x}$ and
$\sigma_{\kappa+i\varepsilon}^{S,x}$ that are Fourier multipliers with
symbols $p^N(\xi;\kappa+i\varepsilon)$ and
$p^S(\xi;\kappa+i\varepsilon)$ respectively: given elements $\phi\in
H^{s+1/2}(\Gamma)$ and $\psi\in H^{s-1/2}(\Gamma)$ and letting
$\hat{\phi}(n)$ and $\hat{\psi}(n)$ denote the Fourier coefficients of
the functions $\phi(x(t))$ and $\psi(x(t))|x'(t)|$, respectively, we
define
\begin{equation}\label{eq:defPS1}
\sigma_{\kappa+i\varepsilon}^{N,x}[\phi](x(t))=\frac{1}{|x'(t)|}\sum_{n\in\mathbb{Z}}p^N(n;\kappa+i\varepsilon)\hat{\phi}(n)e^{int}
\end{equation}
and
\begin{equation}\label{eq:defPS2}
  \sigma_{\kappa+i\varepsilon}^{S,x}[\psi](x(t))=\sum_{n\in\mathbb{Z}}p^S(n;\kappa+i\varepsilon)(n)\hat{\psi}(n)e^{int}.
\end{equation}
It is easy to check that, denoting by $H^{s}[0,2\pi]$ the order-$s$
Sobolev space of $2\pi$-periodic functions in the real
line~\cite{Kress}, $\sigma_{\kappa+i\varepsilon}^{N,x}$ and
$\sigma_{\kappa+i\varepsilon}^{S,x}$ define continuous maps
$\sigma_{\kappa+i\varepsilon}^{N,x}:H^{r+1/2}[0,2\pi]\to
H^{r-1/2}[0,2\pi]$ and
$\sigma_{\kappa+i\varepsilon}^{S,x}:H^{r-1/2}[0,2\pi]\to
H^{r+1/2}[0,2\pi]$. Further, these operators possess the following
positivity relations
\begin{equation}\label{posN}
\Im\ \int_0^{2\pi} \sigma_{\kappa+i\varepsilon}^{N,x}[\phi](x(t))\ \overline{\phi}(x(t))\ |x'(t)|dt >0\quad  \mbox{for every}\quad  \phi\in H^{1/2}[0,2\pi],\ \phi\neq 0
\end{equation}
and
\begin{equation}\label{posS}
  \Im\ \int_0^{2\pi} \sigma_{\kappa+i\varepsilon}^{S,x}[\psi](x(t))\ \overline{\psi}(x(t))\ |x'(t)|dt >0\ \mbox{for every}\quad \psi\in H^{-1/2}[0,2\pi],\quad \phi\neq 0
\end{equation}
which follow from the formulae
\[
\Im\ \int_0^{2\pi} \sigma_{\kappa+i\varepsilon}^{N,x}[\phi](x(t)) \overline{\phi}(x(t))\ |x'(t)|dt=\sum_{n\in\mathbb{Z}}\Im(p^N(n;\kappa+i\varepsilon))|\hat{\phi}(n)|^2
\]
and
\[
\Im\ \int_0^{2\pi} \sigma_{\kappa+i\varepsilon}^{S,x}[\psi](x(t)) \overline{\psi}(x(t))\ |x'(t)|dt=\sum_{n\in\mathbb{Z}}\Im(p^S(n;\kappa+i\varepsilon))|\hat{\psi}(n)|^2
\]
since, by definition, $\Im(p^N(n;\kappa+i\varepsilon)) >0$ and
$\Im(p^S(n;\kappa+i\varepsilon))>0$ for all $n\in\mathbb{Z}$.

Using the principal symbol operators
$\sigma_{\kappa+i\varepsilon}^{N,x}$ and
$\sigma_{\kappa+i\varepsilon}^{S,x}$ we now define a $2\times 2$
regularizing operator $\mathcal{R}^x$ that provides an alternative to
the regularizing operator defined in equation~\eqref{eq:tildR1}. The
entries of the new regularizing operator $\mathcal{R}^x$ are given by
\begin{eqnarray}\label{eq:tildR2}
R_{11}^x&=&\frac{\nu}{1+\nu}I\nonumber\\
R_{12}^x&=&-\frac{2}{1+\nu}\sigma_{\kappa_1}^{S,x}\nonumber\\
R_{21}^x&=&\frac{2\nu}{1+\nu}\sigma_{\kappa_1}^{N,x}\nonumber\\
R_{22}^x&=&\frac{1}{1+\nu}I,
\end{eqnarray}
where $\kappa_1=\kappa+i\varepsilon$. As shown in what follows, use of
this regularizing operator instead of $\mathcal{R}$ in
equation~\eqref{eq:reg_diel} also leads to uniquely solvable integral
equation formulations of the second kind.

In order to state our unique solvability result it is necessary to use
parametric representations of all the integral operators in
equation~\eqref{eq:reg_diel}. These are given by
equation~\eqref{dstarparam} for the operator $D_k^*$, with a similar
expression for the operator $D_k$, and by the expression
$S_{j}^x[\phi](x(t))=\frac{i}{4}\int_0^{2\pi}H_0^{(1)}(k_j|x(t)-x(\tau)|)\phi(x(\tau))|x'(\tau)|d\tau$
for the single layer potential. The corresponding parametric
expression for the hypersingular operator $N_j$ is given
by~\cite{KressH}
$$N_j^x[\psi](x(t))=\frac{1}{4\pi|x'(t)|}\int_0^{2\pi}\cot{\frac{\tau-t}{2}}\psi'(\tau)d\tau+\frac{1}{|x'(t)|}\int_0^{2\pi}M_j(t,\tau)\psi(\tau)d\tau$$ 
where $\psi(\tau)=\psi(x(\tau))$ and where the kernels $M_j(t,\tau)$
are defined as
\begin{equation}\label{eq:defL}
M_j(t,\tau)=\frac{i}{4}k_j^2H_0^{(1)}(k|x(t)-x(\tau)|)x'(t)\cdot x'(\tau)-\tilde{M}_j(t,\tau)\nonumber
\end{equation}
with
\begin{equation}\label{eq:allkernels}
\tilde{M}_j(t,\tau)=\frac{\partial^2}{\partial t\ \partial\tau}\left\{\frac{i}{4}H_0^{(1)}(k_j|x(t)-x(\tau)|)+\frac{1}{4\pi}\ln\left(4\sin^2\frac{t-\tau}{2}\right)\right\}.\nonumber
\end{equation}
As indicated above and established in the following theorem, the
system~\eqref{eq:reg_diel} regularized by means of the principal
symbol operators~\eqref{eq:tildR2} indeed do form a uniquely solvable
system of equations.
\begin{theorem}
Let $\kappa_1=\kappa+i\varepsilon$ where $\kappa>0$ and $\varepsilon>0$. The integral equations
\begin{eqnarray}\label{eq:matrix_PSexplicit}
\left(\begin{array}{cc}\mathcal{B}_{11}&\mathcal{B}_{12}\\\mathcal{B}_{21}&\mathcal{B}_{22}\end{array}\right)\left(\begin{array}{c}a^1\\b^1\end{array}\right)&=&-\left(\begin{array}{c}u^{inc}|_{\Gamma}\\(\partial u^{inc}/\partial n)|_\Gamma\end{array}\right)
\end{eqnarray}
where
\begin{eqnarray}\label{eq:entriesPSexplicit}
\mathcal{B}_{11}&=&\frac{I}{2}+\frac{\nu}{1+\nu}D_1^x-\frac{1}{1+\nu}D_2^x-\frac{2\nu}{1+\nu}(S_1^x+\nu^{-1}S_2^x)\sigma_{\kappa_1}^{N,x}\nonumber\\
\mathcal{B}_{12}&=&\frac{1}{1+\nu}(S_2^x -S_1^x)-\frac{2}{1+\nu}(D_1^x+D_2^x)\sigma_{\kappa_1}^{S,x}\nonumber\\
\mathcal{B}_{21}&=&\frac{\nu}{1+\nu}(N_1^x-N_2^x)-\frac{2\nu}{1+\nu}(D_1^{*,x}+D_2^{*,x})\sigma_{\kappa_1}^{N,x}\nonumber\\
\mathcal{B}_{22}&=&\frac{I}{2}+\frac{\nu}{1+\nu}D^{*,x}_2-\frac{1}{1+\nu}D_1^{*,x}-\frac{2}{1+\nu}(N_1^x+\nu N_2^x)\sigma_{\kappa_1}^{S,x}
\end{eqnarray}
are uniquely solvable in $H^s[0,2\pi]\times H^s[0,2\pi]$ for all
$s\geq \frac{1}{2}$. Furthermore, the matrix operator
$$\mathcal{B}=\left(\begin{array}{cc}\mathcal{B}_{11}&\mathcal{B}_{12}\\
    \mathcal{B}_{21}& \mathcal{B}_{22}\end{array}\right):
H^s[0,2\pi]\times H^s[0,2\pi]\to H^s[0,2\pi]\times H^s[0,2\pi]$$ is a
compact perturbation of the identity matrix.
\end{theorem}
\begin{proof} The proof follows from use of two increased regularity
  results. The first of these, which is established in
  appendix~\ref{app1} (see also~\cite{turc1} for a similar result),
  establishes that for all $s\geq\frac{1}{2}$
  \[
\quad \mbox{``The operator}\
\left(S_{\kappa_1}^x - \sigma_{\kappa_1}^{S,x}\right)\
\mbox{maps}\ H^{s}[0,2\pi]\
\mbox{continuously into}\ H^{s+3}[0,2\pi]\mbox{''}.
\]
To continue with our proof, we produce a corresponding result for the
hypersingular operator $N_{\kappa_1}^x$. To do this we use the
Calder\'on identities~\cite{Nedelec} and we obtain
$$(N_{\kappa_1}^x - \sigma_{\kappa_1}^{N,x})S_{\kappa_1}^x=-\frac{I}{4}+(D_{\kappa_1}^x)^2-\sigma_{\kappa_1}^{N,x}\sigma_{\kappa_1}^{S,x}-\sigma_{\kappa_1}^{N,x}(S_{\kappa_1}^x - \sigma_{\kappa_1}^{S,x}).$$
But, it follows immediately from~\eqref{eq:PS} that
$p^N(\xi;\kappa_1)p^S(\xi;\kappa_1)=-\frac{1}{4}$, and, thus,
from~\eqref{eq:defPS1} and~\eqref{eq:defPS2} we obtain
$\sigma_{\kappa_1}^{N,x}\sigma_{\kappa_1}^{S,x}=-\frac{I}{4}$. It
therefore follows that
$$(N_{\kappa_1}^x - \sigma_{\kappa_1}^{N,x})S_{\kappa_1}^x=(D^x_{\kappa_1})^2-\sigma_{\kappa_1}^{N,x}(S_{\kappa_1}^x - \sigma_{\kappa_1}^{S,x})$$
from which, given the regularizing property of the operator
$S_{\kappa_1}^x - \sigma_{\kappa_1}^{S,x}$, it follows that for all
$s\geq\frac{1}{2}$
\[
\quad \mbox{``The operator}\ \left(N_{\kappa_1}^x -
  \sigma_{\kappa_1}^{N,x}\right)\ \mbox{maps}\ H^{s}[0,2\pi]\
\mbox{continuously into}\ H^{s+1}[0,2\pi]\mbox{''}.
\]
Taking into account the regularizing properties of the operators
$S_{\kappa_1}^x - \sigma_{\kappa_1}^{S,x}$ and $N_{\kappa_1}^x -
\sigma_{\kappa_1}^{N,x}$ we therefore see that the following are
continuous mappings:
\begin{eqnarray}
\mathcal{B}_{11}-\mathcal{A}_{11}&=&-\frac{2\nu}{1+\nu}(S_1^x+\nu^{-1}S_2^x)(\sigma_{\kappa_1}^{N,x}-N_{\kappa_1}^x):H^s[0,2\pi]\to H^{s+2}[0,2\pi]\nonumber\\
\mathcal{B}_{12}-\mathcal{A}_{12}&=&-\frac{2}{1+\nu}(D_1^x+D_2^x)(\sigma_{\kappa_1}^{S,x}-S_{\kappa_1}^x):H^s[0,2\pi]\to H^{s+6}[0,2\pi]\nonumber\\
\mathcal{B}_{21}-\mathcal{A}_{21}&=&-\frac{2\nu}{1+\nu}(D^{*,x}_1+D^{*,x}_2)(\sigma_{\kappa_1}^{N,x}-N_{\kappa_1}^x):H^s[0,2\pi]\to H^{s+4}[0,2\pi]\nonumber\\
\mathcal{B}_{22}-\mathcal{A}_{22}&=&-\frac{2}{1+\nu}(N_1^x+\nu N_2^x)(\sigma_{\kappa_1}^{S,x}-S_{\kappa_1}^x):H^s[0,2\pi]\to H^{s+2}[0,2\pi]\nonumber
\end{eqnarray}
for all $s\geq\frac{1}{2}$. Thus, in view of Theorem~\ref{thm2} we see
that $\mathcal{B}$ is a compact perturbation of the identity matrix in
the space $H^s[0,2\pi]\times H^s[0,2\pi]$. Given the the positivity
relations~\eqref{posS} and~\eqref{posN} for the operators $R_{12}^x$
and $R_{21}^x$, further, it follows from Theorem~\ref{thm3} that the
operator $\mathcal{B}$ is injective.  The claimed invertibility of
$\mathcal{B}$ therefore results from the Fredholm theory, and the
proof is thus complete.
\end{proof}

\section{Numerical results: Nystr\"om discretizations
  \label{prec}} To demonstrate the properties of the integral
equations introduced in this paper vis-a-vis those provided by
classical equations, we introduced high-order Nytr\"om numerical
implementations for all of these formulations. Our integral algorithms
are based on the methodologies put forth
in~\cite{KressH,kusmaul,martensen} for boundary integral operators
with real wavenumbers and their extensions to complex wavenumbers
introduced in~\cite{br-turc,turc1}. These methods, which are based on
use of global trigonometric representation of integral densities
together with explicit quadrature for products of oscillatory
exponentials with weakly singular and singular kernels, will not
discussed here at any length. The implementation of the principal
symbol operators introduced in Section~\ref{reg-II}, in turn, is
straightforward: it amounts to adequate uses of Fast Fourier
Transforms and application of diagonal operators in Fourier
space~\cite{turc1}.

In this section we present a variety of numerical results that
demonstrate the properties of the various formulations. Solutions of
the linear systems arising from the Nystr\"om discretizations were
obtained in all cases by means of the fully complex un-restarted
version of the iterative solver GMRES~\cite{SaadSchultz}. The values
used in each case for the complex wavenumber $\kappa_1$ that enters
the regularizing operators~\eqref{eq:tildR1} and~\eqref{eq:tildR2},
which were selected by means of numerical experimentation to minimize
the number of iterations required by the GMRES iterative solver, are
presented in the captions of the various tables.

We present scattering experiments concerning the following three
smooth scatterers: (a)~a unit circle, (b)~a kite-shaped scatterer
given by $x(t)=(\cos{t}+0.65\cos{2t}-0.65,1.5\sin{t})$~\cite{KressH},
and (c)~a cavity whose parametrization is given by
$x(t)=(x_1(t),x_2(t)),\ x_1(t)=(\cos{t}+2\cos{2t})/2.5,\
x_2(t)=Y(t)/2-Y_s(t)/48$, where $Y(t)=\sin{t}+\sin{2t}+1/2\sin{3t},\
Y_s(t)=-4\sin{t}+7\sin{2t}-6\sin{3t}+2\sin{4t}$; see
Figure~\ref{cavity}. We note that each of these geometries has a
diameter equal to $2$. For all of our experiments we assumed
plane-wave incidence $u^{\rm inc}$: vertical incidence in the
direction of the negative $y$-axis for the circular geometry,
$45^\circ$ incidence pointing into the fourth quadrant for the kite
geometry, and horizontal incidence, pointing towards the positive
$x$-axis, for the cavity structure.
\begin{figure}
\centering
\includegraphics[width=0.55\textwidth]{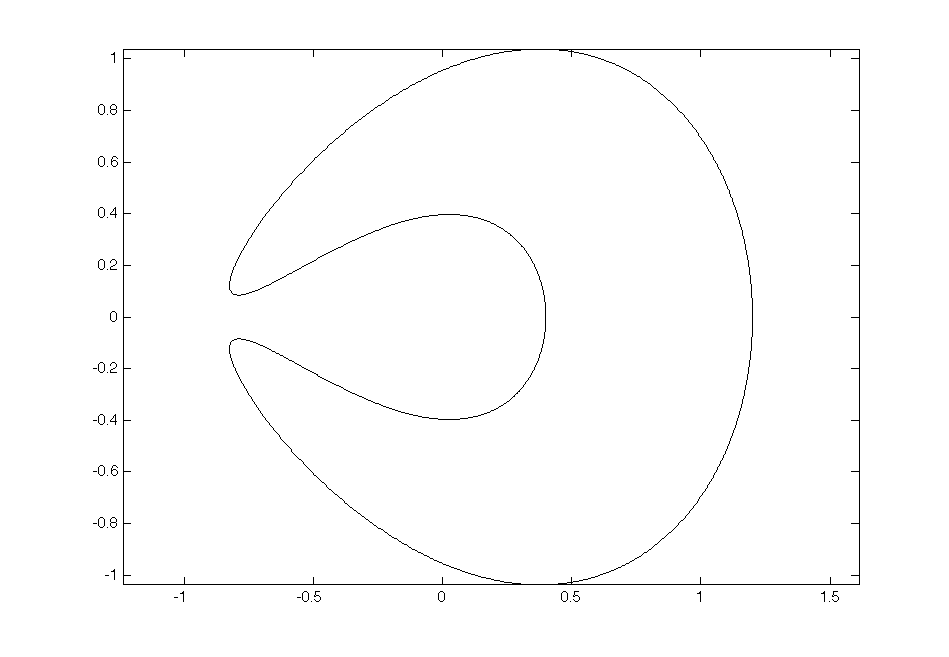}
\caption{Penetrable scatterer of diameter 2 for which the outer region
  $\Omega_1$ contains a cavity-like structure.}
\label{cavity}
\end{figure}

In each case our tables display maximum errors in the far-field
amplitudes. The far-field amplitude $u^{1}_\infty(\hat{\mathbf{x}})$
for each scattering direction
$\hat{\mathbf{x}}=\frac{\mathbf{x}}{|\mathbf{x}|}$ is given by
\begin{equation}
\label{eq:far_field}
u^{1}(\mathbf{x})=\frac{e^{ik_1|\mathbf{x}|}}{\sqrt{|\mathbf{x}|}}\left(u^{1}_\infty(\hat{\mathbf{x}})+\mathcal{O}\left(\frac{1}{|\mathbf{x}|}\right)\right),\
|\mathbf{x}|\rightarrow\infty;
\end{equation}
the far-field error for each scattering direction and each numerical
solution, in turn, was evaluated by means of the relation
\begin{equation}
\label{eq:farField_error}
\varepsilon_\infty={\rm max}|u_\infty^{1,\rm
calc}(\hat{\mathbf{x}})-u_\infty^{1,\rm ref}(\hat{\mathbf{x}})|
\end{equation}
where $u_\infty^{1, \rm calc}$ and $u_\infty^{1,\rm ref} $ denote
numerical and reference far fields, respectively. The reference far
fields $u_\infty^{1,\rm ref}$ were produced using the Mie-series in
the case of circular geometries, and the classical
formulation~\eqref{eq:system_trans} with adequately refined
discretizations and GMRES tolerance set to $10^{-12}$ for all other
geometries. The maximum of the far field errors displayed on our
tables were evaluated through maximization of the far-field error for
a finite but sufficiently large number of scattering directions
$\hat{\mathbf{x}}$.

Besides far field errors, the tables presented in this section display
the numbers of iterations required by the GMRES solver to reach
relative residuals that are specified in each case. Denoting the
underlying frequency by $\omega$ (so that the wavenumber in the $j$-th
dielectric medium $j=1,2$ is given by $k_j =
\omega\sqrt{\varepsilon_j}$) we used discretizations ranging from 4 to
10 discretization points per interior wavelength $\lambda_2 =
2\pi/k_2$, for frequencies $\omega$ in the medium to the
high-frequency range corresponding to acoustic scattering problems of
sizes ranging from $2.5\lambda_2$ to $81.6\lambda_2$. In all cases we
assumed the exterior domain $\Omega_1$ is occupied by vacuum:
$\varepsilon_1 = 1$.  The columns ``Unknowns'' in our tables display
the numbers of unknowns used in each case. (Since all of the systems
of integral equations used involve two unknown integral densities, the
number of unknowns used in each case equals two times the number of
discretization points used.)  The various integral formulations
considered are as follows: the Second Kind
formulations~\eqref{eq:system_trans1}, denoted by
SK~\eqref{eq:system_trans1} and~\eqref{eq:system_trans}
(SK~\eqref{eq:system_trans}); the First Kind
formulation~\eqref{eq:system_trans_FK}
(FK~\eqref{eq:system_trans_FK}), and the Second Kind Regularized
formulations~\eqref{eq:reg_diel}--\eqref{eq:tildR1}
(SKR~\eqref{eq:reg_diel}--\eqref{eq:tildR1})
and~\eqref{eq:reg_diel}--\eqref{eq:tildR2}
(SKR~\eqref{eq:reg_diel}--\eqref{eq:tildR2}).


As it can be seen in Table~\ref{results0}, our transmission solvers
converge with high-order accuracy.
\begin{table}
\begin{center}\resizebox{!}{2.0cm}
{
\begin{tabular}{|c|c|c|c|c|c|c|c|c|c|}
  \hline
  Geometry & Unknowns & \multicolumn{2}{c|}{SK \eqref{eq:system_trans}} &\multicolumn{2}{c|}{FK \eqref{eq:system_trans_FK}}&\multicolumn{2}{c|}{SKR \eqref{eq:reg_diel}--\eqref{eq:tildR1}}& \multicolumn{2}{c|}{SKR \eqref{eq:reg_diel}--\eqref{eq:tildR2}}\\
  \cline{3-10}
  & & Iter.& $\epsilon_\infty$ &Iter.&$\epsilon_\infty$&Iter.&$\epsilon_\infty$&Iter.&$\epsilon_\infty$\\
  \hline
  Circle & 64 & 23 & 2.6 $\times$ $10^{-2}$& 22 & 4.0 $\times$ $10^{-2}$ & 16 & 1.6 $\times$ $10^{-2}$ & 16 & 1.6 $\times$ $10^{-2}$\\
  Circle & 128 & 22 & 6.2 $\times$ $10^{-8}$ & 22 & 6.1 $\times$ $10^{-8}$ & 16 & 6.3 $\times$ $10^{-8}$ & 16 & 5.7 $\times$ $10^{-9}$\\
  \hline
  \hline
  Kite & 128 & 45 & 3.0 $\times$ $10^{-2}$& 47 & 6.7 $\times$ $10^{-2}$ & 31 & 2.8 $\times$ $10^{-2}$ & 33 & 1.9 $\times$ $10^{-2}$\\
  Kite & 192 & 45 & 1.8 $\times$ $10^{-4}$ & 50 & 1.6 $\times$ $10^{-4}$ & 31 & 7.0 $\times$ $10^{-5}$ & 32 & 7.9 $\times$ $10^{-5}$\\
  Kite & 256 & 45 & 4.5 $\times$ $10^{-8}$ & 56 & 1.3 $\times$ $10^{-7}$ & 31 & 1.7 $\times$ $10^{-7}$ & 31 & 1.7 $\times$ $10^{-7}$\\
  \hline
  \hline
  Cavity & 128 & 57 & 9.3 $\times$ $10^{-3}$& 59 & 6.3 $\times$ $10^{-2}$ & 39 & 4.8 $\times$ $10^{-2}$ & 51 & 3.8 $\times$ $10^{-2}$\\
  Cavity & 192 & 58 & 3.5 $\times$ $10^{-5}$ & 64 & 1.6 $\times$ $10^{-4}$ & 39 & 1.1 $\times$ $10^{-4}$ & 50 & 8.9 $\times$ $10^{-5}$\\
  Cavity & 256 & 58 & 4.1 $\times$ $10^{-8}$ & 68 & 1.8 $\times$ $10^{-7}$ & 39 & 1.5 $\times$ $10^{-7}$ & 50 & 1.6 $\times$ $10^{-7}$\\
  \hline
\end{tabular}
}\caption{\label{results0} High-order accuracy of the solvers used
  demonstrated for the three geometries: circle, kite, and cavity. For
  these experiments we used $\nu=1$, $\omega=8$, $\epsilon_1=1$, and
  $\epsilon_2=2$, with GMRES residual tolerance equal to $10^{-8}$.}
\end{center}
\end{table}
Table~\ref{results2}, in turn, presents the computational times
required by the matrix-vector products for each of the four
formulations SK~\eqref{eq:system_trans},
FK~\eqref{eq:system_trans_FK},
SKR~\eqref{eq:reg_diel}--\eqref{eq:tildR1}, and
SKR~\eqref{eq:reg_diel}--\eqref{eq:tildR2} respectively. The results
presented here and throughout this paper were produced by a MATLAB
implementation of the Nystr\"om discretization on a single core of a
MacBookPro computer with a $2.3$ GHz Quad-core Intel i7 processor with
16 GB of memory. We present computational times for the kite geometry
only, as the computational times required by the other geometries
considered in this text are extremely close to those required for the
kite at the same levels of discretization. As it can be seen from the
results in Table~\ref{results2}, the computational times required by a
matrix-vector product for the SK \eqref{eq:system_trans} and FK
\eqref{eq:system_trans_FK} formulations are almost identical. Notably,
the computational times required by a matrix-vector product of the
SKR~\eqref{eq:reg_diel}--\eqref{eq:tildR2} formulation are on average
{\em only} $1.01$ times more expensive than those required by the
classical SK \eqref{eq:system_trans} and FK \eqref{eq:system_trans_FK}
formulations, whereas the computational times required by a
matrix-vector product required by a matrix-vector product of the
SKR~\eqref{eq:reg_diel}--\eqref{eq:tildR1} formulation are on average
at most $1.3$ times more expensive than those required by the
classical SK~\eqref{eq:system_trans} and FK~\eqref{eq:system_trans_FK}
formulations. Given that iteration numbers required by the SKR
formulation can be as small as one-half or even less than the the
corresponding numbers required by the SK and FK formulations (as
demonstrated by Tables~\ref{results21}--\ref{results25}) the new SKR
formulations can lead to savings in computing times by a factor of two
or better---especially in the important high frequency regime.

We note that preconditioners of the FK~\eqref{eq:system_trans_FK}
formulations have been proposed in the literature~\cite{antoineB}:
denoting by $\mathcal{C}$ the matrix operator on the left-hand side of
the first kind equations~(\ref{eq:system_trans_FK}), the idea is to to
use the very operator $\mathcal{C}$ as a preconditioner, since, on
account of Calder\'on's identities~\cite{Nedelec}, the operator
$\mathcal{C}^2$ is a compact perturbation of the identity matrix
operator. We have pursued this avenue, and we have observed that,
although the resulting $\mathcal{C}^2$ formulation converge in half as
many iterations than the $\mathcal{C}$-based formulations, the
$\mathcal{C}^2$ formulations are twice as expensive per iteration.
Thus, that preconditioning strategy does not give rise to any actual
computing-time gains.

\begin{table}
\begin{center}
\begin{tabular}{|c|c|c|c|c|c|}
  \hline
  Geometry & Unknowns & SK \eqref{eq:system_trans}&FK \eqref{eq:system_trans_FK} & SKR \eqref{eq:reg_diel}--\eqref{eq:tildR1} &SKR~\eqref{eq:reg_diel}--\eqref{eq:tildR2}\\
  \hline
  Kite & 512 & 12.99 sec & 12.78 sec& 16.54 sec & 13.97 sec\\ 
  Kite & 1024 & 51.55 sec & 52.34 sec & 66.39 sec & 52.39 sec\\ 
  \hline
\end{tabular}
\caption{\label{results2} Computational times required by a matrix-vector product for each of the four integral equation formulations of the transmission problems considered in this text.}
\end{center}
\end{table} 
 
\begin{table}
\begin{center}\resizebox{!}{1.2cm}
{
\begin{tabular}{|c|c|c|c|c|c|c|c|c|c|c|c|}
\hline
$\omega$ & $\epsilon_1$ & $\epsilon_2$ & Unknowns & \multicolumn{2}{c|}{SK \eqref{eq:system_trans}} &\multicolumn{2}{c|}{FK \eqref{eq:system_trans_FK}}&\multicolumn{2}{c|}{SKR \eqref{eq:reg_diel}--\eqref{eq:tildR1}}& \multicolumn{2}{c|}{SKR \eqref{eq:reg_diel}--\eqref{eq:tildR2}}\\
\cline{5-12}
 & & & & Iter.& $\epsilon_\infty$ &Iter.&$\epsilon_\infty$&Iter.&$\epsilon_\infty$&Iter.&$\epsilon_\infty$\\
\hline
16 & 1 & 2 & 256 & 37 & 7.0 $\times$ $10^{-8}$& 37 & 1.2 $\times$ $10^{-7}$ & 31 & 1.6 $\times$ $10^{-7}$& 31 & 1.6 $\times$ $10^{-7}$\\
32 & 1 & 2 &  512 & 58 & 6.2 $\times$ $10^{-8}$ & 59 & 3.7 $\times$ $10^{-7}$ & 40 & 2.2 $\times$ $10^{-7}$ & 40 & 3.9 $\times$ $10^{-7}$ \\
64 & 1 & 2 & 1024 & 99 & 1.8 $\times$ $10^{-7}$ & 99 & 4.2 $\times$ $10^{-7}$ & 61& 5.0 $\times$ $10^{-7}$& 61 & 4.3 $\times$ $10^{-7}$\\
\hline
\end{tabular}
}
\caption{\label{results21} Scattering experiment for the circular
  geometry with $\nu=1$ and, for the
  SKR~\eqref{eq:reg_diel}--\eqref{eq:tildR1} and
  SKR~\eqref{eq:reg_diel}--\eqref{eq:tildR2} formulations,
  $\kappa_1=(k_1+k_2)/2+ 4i$. GMRES residual was set to equal
  $10^{-8}$.}
\end{center}
\end{table}

\begin{table}
\begin{center}
\resizebox{!}{1.4cm}
{
\begin{tabular}{|c|c|c|c|c|c|c|c|c|c|c|c|}
\hline
$\omega$ & $\epsilon_1$ & $\epsilon_2$ & Unknowns & \multicolumn{2}{c|}{SK \eqref{eq:system_trans}} &\multicolumn{2}{c|}{FK \eqref{eq:system_trans_FK}}&\multicolumn{2}{c|}{SKR \eqref{eq:reg_diel}--\eqref{eq:tildR1}}& \multicolumn{2}{c|}{SKR \eqref{eq:reg_diel}--\eqref{eq:tildR2}}\\
\cline{5-12}
 & & & & Iter.& $\epsilon_\infty$ &Iter.&$\epsilon_\infty$&Iter.&$\epsilon_\infty$&Iter.&$\epsilon_\infty$\\
\hline
16 & 1 & 4 & 512 & 65 & 5.0$\times$ $10^{-4}$& 71 & 1.5 $\times$ $10^{-3}$ & 42 & 1.7 $\times$ $10^{-3}$& 46 & 1.6 $\times$ $10^{-3}$\\
32 & 1 & 4 & 1024 & 93 & 3.1 $\times$ $10^{-3}$ & 104 & 2.0 $\times$ $10^{-3}$ &  52 & 2.6 $\times$ $10^{-3}$ & 62 & 2.6 $\times$ $10^{-3}$ \\
64 & 1 & 4 & 2048 & 128 & 1.1 $\times$ $10^{-3}$ & 138 & 2.3 $\times$ $10^{-3}$ & 64 &  1.7 $\times$ $10^{-3}$ & 74 & 1.6 $\times$ $10^{-3}$\\
128 & 1 & 4 & 4096 & 167 & 1.2 $\times$ $10^{-3}$ & 182 & 2.3 $\times$ $10^{-3}$ & 78 & 1.6 $\times$ $10^{-3}$ & 83 & 1.9 $\times$ $10^{-3}$\\
\hline
\end{tabular}
}
\caption{\label{results22} Scattering experiments for the kite
  geometry with $\nu=1$, and for the SKR~\eqref{eq:reg_diel}--\eqref{eq:tildR1} and
  SKR~\eqref{eq:reg_diel}--\eqref{eq:tildR2} formulations,  $\kappa_1=(k_1+k_2)/2+i\ \omega/4$. GMRES
  residual was set to equal $10^{-4}$.}
\end{center}
\end{table}

\begin{table}
\begin{center}
\resizebox{!}{1.4cm}
{
\begin{tabular}{|c|c|c|c|c|c|c|c|c|c|c|c|}
  \hline
  $\omega$ & $\epsilon_1$ & $\epsilon_2$ & Unknowns & \multicolumn{2}{c|}{SK \eqref{eq:system_trans}} &\multicolumn{2}{c|}{FK \eqref{eq:system_trans_FK}}&\multicolumn{2}{c|}{SKR \eqref{eq:reg_diel}--\eqref{eq:tildR1}}& \multicolumn{2}{c|}{SKR \eqref{eq:reg_diel}--\eqref{eq:tildR2}}\\
  \cline{5-12}
  & & & & Iter.& $\epsilon_\infty$ &Iter.&$\epsilon_\infty$&Iter.&$\epsilon_\infty$&Iter.&$\epsilon_\infty$\\
  \hline
  16 & 1 & 4 & 512 & 111 & 8.2 $\times$ $10^{-4}$& 114 & 2.9 $\times$ $10^{-3}$ & 64 & 4.9 $\times$ $10^{-3}$& 70 & 4.9 $\times$ $10^{-3}$\\
  32 & 1 & 4 & 1024 & 168 & 1.2 $\times$ $10^{-3}$ & 179 & 6.6 $\times$ $10^{-3}$ &  91 & 4.9 $\times$ $10^{-3}$ & 104 & 5.0 $\times$ $10^{-3}$ \\
  64 & 1 & 4 & 2048 & 266 & 1.3 $\times$ $10^{-3}$ & 289 & 2.9 $\times$ $10^{-3}$ & 120 &  3.8 $\times$ $10^{-3}$ & 145 & 3.7 $\times$ $10^{-3}$\\
  128 & 1 & 4 & 4096 & 396 & 1.7 $\times$ $10^{-3}$ & 433 & 3.2 $\times$ $10^{-3}$ & 157 & 3.2 $\times$ $10^{-3}$ & 205 & 3.3 $\times$ $10^{-3}$\\
  \hline
\end{tabular}
}
\caption{\label{results23} Scattering experiments for the cavity geometry, $\nu=1$, and for the SKR~\eqref{eq:reg_diel}--\eqref{eq:tildR1} and SKR~\eqref{eq:reg_diel}--\eqref{eq:tildR2} formulations, $\kappa_1=(k_1+k_2)/2+i\ \omega/4 $. GMRES residual was set to equal $10^{-4}$.}
\end{center}
\end{table}

\begin{table}
\begin{center}
\resizebox{!}{1.4cm}
{
\begin{tabular}{|c|c|c|c|c|c|c|c|c|c|c|c|}
\hline
$\omega$ & $\epsilon_1$ & $\epsilon_2$ & Unknowns & \multicolumn{2}{c|}{SK \eqref{eq:system_trans}} &\multicolumn{2}{c|}{FK \eqref{eq:system_trans_FK}}&\multicolumn{2}{c|}{SKR \eqref{eq:reg_diel}--\eqref{eq:tildR1}}& \multicolumn{2}{c|}{SKR \eqref{eq:reg_diel}--\eqref{eq:tildR2}}\\
\cline{5-12}
 & & & & Iter.& $\epsilon_\infty$ &Iter.&$\epsilon_\infty$&Iter.&$\epsilon_\infty$&Iter.&$\epsilon_\infty$\\
\hline
8 & 1 & 16 & 512 & 79 & 1.4 $\times$ $10^{-3}$& 210 & 2.4 $\times$ $10^{-3}$ & 65 & 1.8 $\times$ $10^{-3}$& 66 & 2.0 $\times$ $10^{-3}$\\
16 & 1 & 16 & 1024 & 122 & 5.0 $\times$ $10^{-3}$ & 283 & 6.8 $\times$ $10^{-3}$ & 97 & 5.6 $\times$ $10^{-3}$ & 91 & 5.5 $\times$ $10^{-3}$ \\
32 & 1 & 16 & 2048 & 176 & 7.8 $\times$ $10^{-3}$ & 373 & 3.0 $\times$ $10^{-3}$ & 112 &  2.2 $\times$ $10^{-3}$ & 109 & 1.9 $\times$ $10^{-3}$\\
64 & 1 & 16 & 4096 & 263 & 9.1 $\times$ $10^{-4}$ & 497 & 3.2 $\times$ $10^{-3}$ & 147 & 1.9 $\times$ $10^{-3}$ & 147 & 2.6 $\times$ $10^{-3}$\\
128 & 1 & 16 & 8192 & 338 & 7.7 $\times$ $10^{-4}$ & 649  & 3.0 $\times$ $10^{-3}$ & 187 & 2.1 $\times$ $10^{-3}$ & 187 & 2.2 $\times$ $10^{-3}$\\
\hline
\end{tabular}
}
\caption{\label{results24} Scattering experiments for the kite
  geometry with $\nu=\epsilon_1/\epsilon_2$, and for the SKR~\eqref{eq:reg_diel}--\eqref{eq:tildR1} and
  SKR~\eqref{eq:reg_diel}--\eqref{eq:tildR2} formulations, $\kappa_1=(k_1+k_2)/2+i\
  \omega $. GMRES
  residual was set to equal $10^{-4}$. Through additional experiments we have
  determined that the formulations SK~\eqref{eq:system_trans1} require
  iterations numbers equal to 129, 207, 288, 318, and 393 for the
  cases when $\omega=8,16,32,64$ and $128$, respectively.}
\end{center}
\end{table}

\begin{table}
\begin{center}
\resizebox{!}{1.4cm}
{
\begin{tabular}{|c|c|c|c|c|c|c|c|c|c|c|c|}
  \hline
  $\omega$ & $\epsilon_1$ & $\epsilon_2$ & Unknowns & \multicolumn{2}{c|}{SK \eqref{eq:system_trans}} &\multicolumn{2}{c|}{FK \eqref{eq:system_trans_FK}}&\multicolumn{2}{c|}{SKR \eqref{eq:reg_diel}--\eqref{eq:tildR1}}& \multicolumn{2}{c|}{SKR \eqref{eq:reg_diel}--\eqref{eq:tildR2}}\\
  \cline{5-12}
  & & & & Iter.& $\epsilon_\infty$ &Iter.&$\epsilon_\infty$&Iter.&$\epsilon_\infty$&Iter.&$\epsilon_\infty$\\
  \hline
  8 & 1 & 16 & 512 & 114 & 1.3 $\times$ $10^{-3}$& 246 & 8.7 $\times$ $10^{-3}$ & 85 & 9.0 $\times$ $10^{-3}$& 85 & 8.8 $\times$ $10^{-3}$\\
  16 & 1 & 16 & 1024 & 182 & 2.7 $\times$ $10^{-3}$ & 429 & 1.6 $\times$ $10^{-2}$ & 148 & 1.6 $\times$ $10^{-2}$ & 148 & 1.5 $\times$ $10^{-3}$ \\
  32 & 1 & 16 & 2048 & 341 & 3.6 $\times$ $10^{-3}$ & 661 & 2.1 $\times$ $10^{-2}$ & 200 &  2.1 $\times$ $10^{-2}$ & 202 & 2.0 $\times$ $10^{-2}$\\
  64 & 1 & 16 & 4096 & 489 & 3.1 $\times$ $10^{-3}$ & 1094 & 2.9 $\times$ $10^{-3}$ & 278 & 3.4 $\times$ $10^{-3}$ & 297 & 2.1 $\times$ $10^{-3}$\\
  128 & 1 & 16 & 8192 & 877 & 7.0 $\times$ $10^{-4}$ & 1560 & 2.3 $\times$ $10^{-3}$ & 397 & 1.9 $\times$ $10^{-3}$ & 406  & 1.6 $\times$ $10^{-3}$\\

\hline
\end{tabular}
}
\caption{\label{results25} Scattering experiments for the cavity
  geometry, $\nu=\epsilon_1/\epsilon_2$, for the SKR~\eqref{eq:reg_diel}--\eqref{eq:tildR1} and
  SKR~\eqref{eq:reg_diel}--\eqref{eq:tildR2} formulations, $\kappa_1=(k_1+k_2)/2+i\
  \omega $. GMRES residual was set to equal $10^{-4}$.  Through additional experiments we have
  determined that the formulations SK~\eqref{eq:system_trans1} require
  iterations numbers equal to 167, 290, 481, 663, and 1232
  respectively for the cases when $\omega=8,16,32,64$ and $128$.}
\end{center}
\end{table}

\section*{Acknowledgments}
Oscar Bruno thanks NSF and AFOSR for their support during the
preparation of this work. Yassine Boubendir gratefully acknowledges
support from NSF through contract DMS-1016405. Catalin Turc gratefully
acknowledge support from NSF through contract DMS-1008076.

\appendix
\section{Positivity of complex wavenumber operators}\label{app1}

In this section we establish certain positivity properties for the
operators $S_{\kappa+i\varepsilon}$ and $N_{\kappa + i\epsilon}$
(cf.~\cite{Nedelec}).
\begin{lemma}\label{sixp1} Let $\kappa>0$ and $\varepsilon>0$. Then
$$\Im\ \int_\Gamma S_{\kappa+i\varepsilon} [\phi]\ \overline{\phi}\ ds >0,\quad {\rm for\ all}\ \phi\in H^{-1/2}(\Gamma),\ \phi\neq 0.$$
\end{lemma}
\begin{proof} Let
$$U(z)=S_{\kappa+i\varepsilon}[\phi](z)=\int_\Gamma G_{\kappa+i\varepsilon}(z-y)\phi(y)ds(y),\ z\in \mathbb{R}^2\setminus \Gamma$$
and call $U^{+}$ and $U^-$ the restrictions of $U$ to $\Omega_1$ and
$\Omega_2$, respectively. Given that $U^-$ is a solution of the
Helmholtz equation with wavenumber $\kappa+i\epsilon$ in the domain
$\Omega_2$ we obtain
$$\int_\Gamma U^{-}\ \overline{\partial_n U^{-}}ds = \int_{\Omega_2}(|\nabla U^{-}|^2-(\kappa-i\varepsilon)^2|U^{-}|^2)dx$$
from which it follows that
$$\Im\ \int_\Gamma U^{-}\ \overline{\partial_{n} U^{-}}ds = 2\kappa\varepsilon\int_{\Omega_2} |U^{-}|^2dx.$$
Let $R>0$ be such that $\Omega_2$ is contained in $B_R=\{x:\|x\|_2\leq
R\}$. An application of the Green's formula in the domain
$\Omega_1\bigcap B_R$ gives
$$\int_{\partial B_R}U^{+}\ \overline{\partial_r U^{+}}ds - \int_\Gamma U^{+}\ \overline{\partial_n U^{+}}ds=\int_{\Omega_1\bigcap B_R}(|\nabla U^{+}|^2-(\kappa-i\varepsilon)^2|U^{+}|^2)dx$$
from which we obtain
$$
\Im\ \int_{\partial B_R}U^{+}\ \overline{\partial_r U^{+}}ds - \Im\
\int_\Gamma U^{+}\ \overline{\partial_n U^{+}}ds=2\kappa\varepsilon
\int_{\Omega_1\bigcap B_R}|U^{+}|^2dx.$$ Given that $\varepsilon >0$
and the resulting exponential decay of $U^{+}(x)$ as $|x|\to \infty$
it follows that $\lim_{R\to\infty}\int_{\partial B_R}U^{+}\
\overline{\partial_r U^{+}}ds=0$, and, thus
\begin{equation}\label{all-space}
\Im\ \int_\Gamma U^{-}\ \overline{\partial_n U^{-}}ds -\Im\ \int_\Gamma U^{+}\ \overline{\partial_n U^{+}}ds= 2\kappa\varepsilon(\int_{\Omega_2} |U|^2dx+\int_{\Omega_1}|U|^2dx).
\end{equation} 
Since $U^{+}=U^{-}= S_{\kappa+i\varepsilon}[\phi]$ and $\partial_n
U^{-}-\partial_n U^{+} =\phi$ on $\Gamma$ we see that
$$\Im\ \int_\Gamma S_{\kappa+i\varepsilon} [\phi]\ \overline{\phi}\ ds =2\kappa\varepsilon\int_{\mathbb{R}^2}|U|^2dx$$
and the Lemma follows.
\end{proof}

\begin{lemma}\label{sixp2} Let $\kappa>0$ and $\varepsilon>0$. Then
$$\Im\ \int_\Gamma N_{\kappa+i\varepsilon} [\psi]\ \overline{\psi}\ ds >0,\quad {\rm for\ all}\ \psi\in H^{1/2}(\Gamma),\ \psi\neq 0.$$
\end{lemma}
\begin{proof} Let
$$V(z)=D_{\kappa+i\varepsilon}[\psi](z)=\int_\Gamma \frac{\partial G_{\kappa+i\varepsilon}(z-y)}{\partial n(y)}\psi(y)ds(y),\ z\in \mathbb{R}^2\setminus \Gamma.$$
Given that $V$ is a solution of the Helmholtz equation with wavenumber
$\kappa+i\epsilon$ in the domain $\Omega_2$ and a radiative solution
of the Helmholtz equation with wavenumber $\kappa+i\epsilon$ in the
domain $\Omega_1$, as in the derivation of equation~\eqref{all-space}
we obtain
$$\Im\ \int_\Gamma \partial_nV^{+}\ \overline{V^{+}}ds -\Im\ \int_\Gamma \partial_n V^{-}\ \overline{V^{-}}ds= 2\kappa\varepsilon(\int_{\Omega_2} |V|^2dx+\int_{\Omega_1\bigcap B_R}|V|^2dx).$$
Since $V^{+}|_\Gamma-V^{-}|_\Gamma =\psi$ and $\partial_n
V^{+}=\partial_n V^{-}=N_{\kappa+i\varepsilon}[\psi]$ on $\Gamma$ we
obtain
$$\Im\ \int_\Gamma N_{\kappa+i\varepsilon} [\psi]\ \overline{\psi}\ ds =2\kappa\varepsilon \int_{\mathbb{R}^2} |V|^2dx$$
from which the result of the Lemma follows.
\end{proof}
 
\section{Increased regularity of the operator $S_{\kappa+i\varepsilon}^x -\sigma^{S,x}_{\kappa+i\varepsilon}$}\label{app2}
This appendix establishes that, given a smooth $2\pi$ periodic
parametrization $x:[0,2\pi]\to\Gamma$ of the curve $\Gamma$, the
difference between the parametric form
$S_{\kappa+i\varepsilon}^x[\psi](x(t))=\frac{i}{4}\int_0^{2\pi}H_0^{(1)}((\kappa+i\varepsilon)|x(t)-x(\tau)|)\psi(x(\tau))|x'(\tau)|d\tau$
of the single layer operator and with wavenumber $\kappa_1 =
\kappa+i\varepsilon$ and the Fourier operator
$\sigma^{S,x}_{\kappa+i\varepsilon}$ defined in
equation~\eqref{eq:defPS2} is a three-order smoothing operator.
\begin{lemma} For all $s\geq\frac{1}{2}$ the operator
  $S_{\kappa+i\varepsilon}^x - \sigma^{S,x}_{\kappa+i\varepsilon}$
  maps $H^{s}[0,2\pi]$ continuously into $H^{s+3}[0,2\pi]$.
\end{lemma}
\begin{proof} Using the small $z$ expression~\cite{Abramowitz}
$$\frac{i}{4}H_0^{(1)}(k|z|)=-\frac{1}{2\pi}\log{|z|}+\frac{i}{4}-\frac{1}{2\pi}\left(\log{\frac{k}{2}}+C\right)+\mathcal{O}(|z|^2\log{|z|}),\quad z\to 0$$
where $C$ is Euler's constant, we see that the kernel of the operator
$S_{\kappa+i\varepsilon}^x$ can be expressed in the form
\begin{eqnarray}
\label{eq:expression_kernel}
\frac{i}{4}H_0^{(1)}((\kappa+i\varepsilon)|x(t)-x(\tau)|)|x'(\tau)| &=&\frac{i}{4}H_0^{(1)}((\kappa+i\varepsilon)|e^{it}-e^{i\tau}|)|x'(\tau)| + \Phi_1(t,\tau)\nonumber\\
&+&\Phi_2(t,\tau)
\end{eqnarray}
where $\Phi_1(t,\tau)$ is a smooth function of $(t,\tau)$ and
$\Phi_2(t,\tau)=\mathcal{O}(|f(t,\tau)|^2\log{|f(t,\tau)|})$ as
$|f(t,\tau)|\to 0$, and where $f(t,\tau)=t-\tau\ {\rm mod}\ 2\pi$. It
follows that for given $\psi\in H^s[0,2\pi]$ we may write
\[
S_{\kappa+i\varepsilon}^x[\psi](x(t))=S_{\kappa+i\varepsilon}^{\mathbb{S}^1,x}[\psi](t)+S_{\Phi_1}^x[\psi](x(t))+S_{\Phi_2}^x[\psi](x(t))
\]
where
\[
S_{\kappa+i\varepsilon}^{\mathbb{S}^1,x}[\psi](t)=\frac{i}{4}\int_0^{2\pi}H_0^{(1)}((\kappa+i\varepsilon)|e^{it}-e^{i\tau}|)\psi(x(\tau))|x'(\tau)|d\tau
\]
and
\[
S_{\Phi_j}^x[\psi](x(t))=\int_0^{2\pi}\Phi_j(t,\tau)\psi(x(\tau))d\tau,\ j=1,2.
\]
Given that the kernel $\Phi_1(t,\tau)$ is a smooth function of
$(t,\tau)$ and taking into account the behavior of the kernel
$\Phi_2(t,\tau)$ as $|f(t,\tau)|\to 0$, we see that
$S_{\Phi_j}^x:H^s[0,2\pi]\to H^{s+3}[0,2\pi],\ j=1,2$. It thus
suffices to show that the difference
$S_{\kappa+i\varepsilon}^{\mathbb{S}^1,x}[\psi](x(t)) -
S_{\kappa+i\varepsilon}^x[\psi](x(t))$ defines an operator that maps
$H^{s}[0,2\pi]$ continuously into $H^{s+3}[0,2\pi]$.


To do this we use the Fourier expansion
$\psi(x(\tau))|x'(\tau)|=\sum_{n\in\mathbb{Z}}\hat{\psi}(n)e^{in\tau}$
and we
obtain $$\frac{i}{4}\int_0^{2\pi}H_0^{(1)}((\kappa+i\varepsilon)|e^{it}-e^{i\tau}|)\psi(x(\tau))|x'(\tau)|d\tau=\sum_{n=-\infty}^{\infty}\hat{\psi}(n)\left(\frac{i}{4}\int_0^{2\pi}H_0^{(1)}((\kappa+i\varepsilon)|e^{it}-e^{i\tau}|)e^{in\tau}d\tau\right),$$
so that, in view of well known relations for circular
scatterers~\cite{Kress1985} there follows
\begin{equation}\label{product}
  \frac{i}{4}\int_0^{2\pi}H_0^{(1)}((\kappa+i\varepsilon)|e^{it}-e^{i\tau}|)e^{in\tau}d\tau=\frac{i\pi}{2}J_{|n|}(\kappa+i\varepsilon)H_{|n|}^{(1)}(\kappa+i\varepsilon)e^{in t}.
\end{equation}
Now, let
$=s_{|n|}(\kappa+i\varepsilon)=\frac{i\pi}{2}J_{|n|}(\kappa+i\varepsilon)H_{|n|}^{(1)}(\kappa+i\varepsilon)$.
Invoking the representation of the functions
$J_{|n|}(\kappa+i\varepsilon)$ and
$H_{|n|}^{(1)}(\kappa+i\varepsilon)$ in terms of the Bessel functions
of the third kind, so that $\frac{i\pi
}{2}J_{|n|}(\kappa+i\varepsilon)H_{|n|}^{(1)}(\kappa+i\varepsilon)=I_{|n|}(\varepsilon-i\kappa)K_{|n|}(\varepsilon-i\kappa)$,
together with the uniform asymptotic expansions as
$\nu\rightarrow\infty$ and $|\arg{z}|\leq \frac{\pi}{2}-\delta$,
\begin{eqnarray}\label{eq:estimateK}
I_\nu(\nu z)&\sim&\frac{1}{\sqrt{2\pi\nu}}\frac{e^{\nu\mu}}{(1+z^2)^\frac{1}{4}}(1+u_1(t)\nu^{-1}+\mathcal{O}(\nu^{-2}))\nonumber\\
K_\nu(\nu z)&\sim&\frac{\sqrt{\pi}}{\sqrt{2\nu}}\frac{e^{-\nu\mu}}{(1+z^2)^\frac{1}{4}}(1-u_1(t)\nu^{-1}+\mathcal{O}(\nu^{-2}))
\end{eqnarray}
where $\mu=\sqrt{1+z^2}+\ln\frac{z}{1+\sqrt{1+z^2}}$,
$t=\frac{1}{\sqrt{1+z^2}}$, and $u_1(t)=\frac{3t-5t^3}{24}$ (Formulas
9.7.7 and 9.7.8 in~\cite{Abramowitz}) we obtain 
$$s_{|n|}(\kappa+i\varepsilon)=\frac{1}{2\sqrt{n^2+(\varepsilon-i\kappa)^2}}\left(1+\mathcal{O}(|n|^{-2})\right)=p^S(n;\kappa+i\varepsilon)+\mathcal{O}(|n|^{-3}),\quad |n|\to\infty.$$
The last identity implies that
$S_{\kappa+i\varepsilon}^{\mathbb{S}^1,x}-\sigma^{S,x}_{\kappa+i\varepsilon}$
maps $H^s[0,2\pi]$ continuously into $H^{s+3}[0,2\pi]$, and the proof
of the lemma is thus complete.\end{proof}

\end{document}